\documentclass{amsart}
\usepackage{packages}

\AtBeginDocument{{\noindent\small
{\em To appear in: J. Fractal Geometry (2023)}}
\vspace{9mm}}

\begin{document}

\title[3D Koch-type crystals]{3D Koch-type crystals}
\author{Giovanni Ferrer}
\author{Alejandro V\'elez\,-\,Santiago}

\address{Giovanni Ferrer\hfill\break
Department of Mathematics\\ The Ohio State University, 100 Math Tower, 231 West 18th Avenue, Columbus, OH, 43210, U.S.A.}
\email{ferrer.40@osu.edu}

\address{Alejandro V\'elez\,-\,Santiago\hfill\break
Department of Mathematical Sciences\\ University of Puerto Rico at Mayag\"uez\\ Mayag\"uez, PR 00681}
\email{alejandro.velez2@upr.edu,\,\,\,dr.velez.santiago@gmail.com}


%\begin{abstract}
%\input{version09July17}
%%This is a sample document which shows the most important features of the AMS Journal

\begin{abstract} We consider the construction of a family $\{K_N\}$ of $3$-dimensional Koch-type surfaces, with a corresponding family of $3$-dimensional Koch-type ``snowflake analogues" $\{\mathcal{C}_N\}$, where $N>1$ are integers with $N \not\equiv 0 \,(\bmod\,\, 3)$. We first establish that the Koch surfaces $K_N$ are $s_N$-sets with respect to the $s_N$-dimensional Hausdorff measure, for $s_N=\log(N^2+2)/\log(N)$ the Hausdorff dimension of each Koch-type surface $K_N$. Using self-similarity, one deduces that the same result holds for each Koch-type crystal $\mathcal{C}_N$. We then develop lower and upper approximation monotonic sequences converging to the $s_N$-dimensional Hausdorff measure on each Koch-type surface $K_N$, and consequently, one obtains upper and lower bounds for the Hausdorff measure for each set $\mathcal{C}_N$. As an application, we consider the realization of Robin boundary value problems over the Koch-type crystals $\mathcal{C}_N$, for $N>2$.
\end{abstract}

\maketitle

\section{Introduction}\label{sec1}

\indent The aim of this paper is to give rise to  $3$-dimensional Koch-type fractal sets which exhibit some analogies in some sense to both the Koch curve and the Koch snowflake. These $3$-dimensional fractal sets will be called Koch $N$-surfaces and Koch $N$-crystals, respectively (see Section \ref{sec3} for illustrations and precise definitions of these sets). Although the geometry of these sets and the corresponding pre-fractal sets may have been considered and visualized, in our knowledge, there is no concrete mathematical construction and analysis of Koch-type surfaces and Koch-type crystals, up to the present time. Using geometric and self-similarity tools, we deduce the generation of a family of compact invariant self-similar sets, which correspond precisely to Koch $N$-surfaces $K_N$ (for $N\in\mathbb{N\!}\,$ with $N \not\equiv 0 \,(\bmod\,\, 3)$). From here, using standard methods as in \cite{falconer1,HUT81,MAT}, we compute the Hausdorff dimension $s_N$ of each Koch $N$-surface $K_N$, and obtain that $\{K_N\}$ form a family of $s_N$-set with respect to the $s_N$-dimensional Hausdorff measure. The self-similar properties of each $K_N$ lead to the construction of a family of Koch $N$-crystals $\{\mathcal{C}_N\}$, whose boundaries (in the case $N>2$) are also $s_N$-sets with respect to the same values $s_N$ and same measures. In particular, when $N>2$ with $N \not\equiv 0 \,(\bmod\,\, 3)$, the crystals $\{\mathcal{C}_N\}$ can be regarded as a family of open connected domains with Koch-type fractal boundaries. This plays an important role in certain applications, which we will consider at the end of the paper.\\
\indent We then generalize tools developed by Jia \cite{bao1,bao2} (for $2$-dimensional fractals) to establish the main results of the paper, which consist on approximating the $s_N$-dimensional Hausdorff measure of each Koch $N$-surface $K_N$ by means of increasingly precise upper and lower bounds. To be more precise, we will establish the existence of a decreasing sequence $\{a_n(N)\}$ of positive numbers, and an increasing sequence $\{a'_n(N)\}$ of positive numbers, such that\\
\begin{equation}\label{MR}
a'_n(N)\,\leq\,\hh^{s_N}(K_N)\,\leq\,a_n(N),\,\,\,\textrm{for each}\,\,\,n\in\mathbb{N\!}\,,\,\,\,\,\,\textrm{and}\,\,\,\,\,
\displaystyle\limsup_{n\rightarrow\infty}a'_n(N)=\hh^{s_N}(K_N)=\displaystyle\liminf_{n\rightarrow\infty}a_n(N).
\end{equation}
Some applications to boundary value problems over the family $\{\mathcal{C}_N\}$ of Koch $N$-crystals will be addressed.\\
\indent Fractals play a role in many areas in Mathematics, with multiple applications to other fields. Concerning Koch-type fractal sets, there is a vast amount of research done over the classical Koch snowflake domain (see image below).
\begin{center}
\begin{tikzpicture}
\shadedraw[shading=color wheel]
[l-system={rule set={F -> F-F++F-F}, step=0.9pt, angle=60,
   axiom=F++F++F, order=4}] lindenmayer system -- cycle;
\end{tikzpicture}
\end{center}
\begin{center}
\textsc{Figure 0:} The Koch snowflake domain
\end{center}
In particular, the fact that the interior of the Koch snowflake domain is an open connected set, and the boundary is a self-similar $d$-set (for $d=\log(4)/\log(3)$), has allowed the well posedness and regularity results for boundary value problems over such region. One can refer to the works in \cite{LAN-VELEZ-VER18-1,LAN-VELEZ-VER15-1,LAP-PANG95,VELEZ2013-1} (among many others). The interior of the Koch snowflake is an example of a finitely connected $(\varepsilon,\delta)$-domain (e.g. Definition \ref{d2.5}), which in views of \cite{JON} is equivalent to say that the interior of the domain satisfies the $p$-extension property in the sense of \cite[pag. 1]{JON} (also called a Jones domain). It is important to point out that the exact value of the of the $d$-Hausdorff measure for the classical Koch snowflake (refer to Figure 0) is unknown, up to the present time. Approximation sequences fulfilling a statement as in (\ref{MR}) were developed by Jia \cite{bao2}, and this work motivates the generalization to the 3D case, which is the heart of the present paper.\\
\indent In the case of $3$-dimensional domains, the equivalence provided by \cite{JON} for finitely connected Jordan curves in $\mathbb{R\!}^{\,2}$ is no longer valid. Furthermore, there is little literature concerning domains in $\mathbb{R\!}^{\,3}$ with fractal boundaries that may exhibit sufficient geometric properties, allowing the interior to be an $(\varepsilon,\delta)$-domain, and the boundary to be a $d$-set. However, such domains in $\mathbb{R}^n$ that can be constructed via natural polyhedral approximations are indeed of interest, and have been considered in \cite[\S 6]{Teplyaev2022}. Thus, motivated from the structure and construction of the Koch snowflake domain, we have assembled a family of $3$-dimensional connected domains whose fractal boundaries can be viewed as the limit of a sequence of pre-fractal sets (which are Lipschitz) having similar structure as the Koch curve. It follows that many of the properties of the snowflake domain are inherited by the Koch-type surfaces and crystals, which opens the door for multiple extensions and applications. In particular, one can define partial differential equations over the interior of the Koch $N$-crystals, and obtain solvability and regularity results. These latter applications will be discussed in more detail in Section \ref{sec7}.\\
\indent The paper is organized in the following way. Section \ref{sec2} provides an overview of the basic concepts, definitions and results concerning self-similar sets and the geometry of domains. In Section \ref{sec3}, we give a precise definitions and constructions for the Koch $N$-surfaces $K_N$, and the existence of a family $\{\mathcal{C}_N\}$  of Koch crystals. Geometrical motivations and justifications are also provided. At the end, we show that each Koch $N$-surface is a $s_N$-set with respect to the $s_N$-dimensional Hausdorff measure, for $s_N=\log(N^2+2)/\log(N)$.
In Section \ref{sec4}, we provide all the machinery needed to provide concrete definitions for the sequences $\{a_n\}$ and $\{a'_n\}$ mentioned in the previous paragraphs, and we state the main results of the paper, which consists in the fulfillment of (\ref{MR}). Some more general useful results are also established in this section, whose validity extend to more general classes of fractal self-similar sets. Section \ref{sec5} is purely devoted to the proof of the main result of the paper for the particular case $N=2$, while Section \ref{sec6} takes care of the proof of the main result (\ref{MR}) when $N>2$. Finally, Section \ref{sec7} presents an example of a linear partial differential equation with Robin boundary conditions over the Koch $N$-crystals, for $N>2$. We show that the structure of these crystals, which can be viewed as domains with fractal boundaries, allows the Robin problem to be well posed, solvable, and with fine regularity results.

\section{Preliminaries}\label{sec2}

\indent In this section, we collect some basic definitions and results who will play a role in the subsequent sections.
\begin{definition}\label{d2.0}
We denote the \textbf{Hausdorff distance} of $A,B \subset \R^n$ by 
$$d_H(A,B) := \max \left\{ \sup_{a \in A}\, \inf_{b \in B} \|a - b\| ,\; \sup_{b \in B} \,\inf_{a \in A} \|a - b\| \right\},$$
where $\|\cdot\|$ the denotes the euclidean norm on $\R^n$.
Furthermore, we will denote the \emph{diameter} of $C \subset \R^n$ by
$$|C| := \sup_{c_1,c_2 \in A} \|c_1 - c_2\|.$$
\end{definition} 
 
\begin{definition}\label{d2.1} A mapping $S: \R^n \rightarrow \R^n$ is called a \textbf{similitude} if there exists $0 < r < 1$, such that
$$|S(x) - S(y)| = r|x-y| \text{, for } x,y \in \R^n.$$
Similitudes are exactly those maps $S$ which can be written as
$$S(x) = rg(x) + z \text{, for } x \in \R^n,$$
for some $g \in O(n)$, $z \in \R^n$ and $0 < r < 1$. We say that $r$ is the contraction ratio of $S$.
\end{definition}

\begin{definition}\label{d2.2} Let $S = \{S_1, \ldots, S_M\}$ ($M \geq 2$) be a finite sequence of similitudes with contraction ratios $\{r_1,\ldots,r_{M}\}$ ($0<r_i<1$).
\begin{enumerate}
\item[(a)]\,\,\,We say that a non-empty compact set $K$ is \textbf{invariant} under $S$, if
$$K = \bigcap_{i = 1}^M S_iK.$$
\item[(b)]\,\,\,If in addition, $$\hh^s(S_i(K) \cap S_j(k)) = 0 \text{, for } i \neq j,\,\,\,\,\,\,\,\,\,\textrm{for}\,\,\,s=\textrm{dim}_{\hh}(K),$$ then we call the invariant set $K$ \textbf{self-similar}.
\item[(c)]\,\,\,The \textbf{similarity dimension} of $K$ is defined as the unique $s\geq0$, such that $$\displaystyle\sum^M_{i=1}r^s_i=1.$$
\end{enumerate}
\end{definition}

\indent In views of \cite{falconer1}, it is known that for any such $S$, there exists a unique invariant compact set. 

\begin{definition}\label{2.3} We say that a family of similitudes $S = \{S_1, \hdots, S_M\}$ ($M \geq 2$) satisfies the \textbf{open set condition} if there exists a non-empty open set $V$ such that
$$\bigcup_{i = 1}^M S_i(V) \subset V,\,\,\,\,\,\text{ and }\,\,\,\,\,\,S_i(V) \cap S_j(V) = \emptyset\,\,\, \text{whenever}\,\,i \neq j.$$
\end{definition}

\begin{definition}\label{2.4} Let $K \subset \R^n$ be a compact set, $s \in [0, n]$, and $\mu$ a positive measure supported $K$. We say that $K$ is a $s$\textbf{-set} with respect to the measure $\mu$, if there exist constants $a,\,b,\,R > 0$, such that      
$$ar^s \leq \mu(K \cap B(x,r)) \leq br^s,\,\,\,\, \text{for all }\,\,x \in K, \quad 0 < r \leq R.$$
In this case, we call $\mu$ an $s$\textbf{-Ahlfors measure} on $K$.
\end{definition} 

\indent The following result is important.

\begin{theorem}\label{T1}\,(see \cite{HUT81,MAT})\, If the family $S = \{S_1, \hdots, S_M\}$ with contraction ratios $r_1, \hdots, r_M$ satisfies the open set condition, then the invariant compact set $K$ under $S$ is self-similar, with $0 < \hh^s(K) < \infty$, for $s = \text{dim}_{\hh} K$. Furthermore, $s$ equals the similarity dimension of $K$, and $K$ is a $s$-set with respect to $\hh^s$. 
\end{theorem}

\indent We conclude this section with the following geometric definition of a domain, introduced by Jones \cite{JON}.

\begin{definition}\label{d2.5}
An open set $\Omega\subseteq\mathbb{R}^n$ is called an\,
($\varepsilon,\,\delta$)\textbf{-domain}, if there exists $\delta\in (0,+\infty]$ and there exists $\varepsilon\in (0,1]$, such that for each $x,\,y\in\Omega$ with $|x-y|\leq\delta$, there exists a continuous rectifiable curve
$\gamma :[0,t]\rightarrow\Omega$, such that $\gamma(0)=x$ and $\gamma(t)=y$, with the following properties:
\begin{itemize}
    \item[(i)] $l(\{\gamma\})\leq\frac{1}{\varepsilon}|x-y|$.
    \item[(ii)] $\textrm{dist}(z,\partial\Omega)\geq \frac{\varepsilon |x-z| \, |y - z|}{|x - y|}$ for all $z$ on $\gamma$.
\end{itemize}
Also, an $(\varepsilon,\,\infty)$-domain is called an \textbf{uniform domain}.
\end{definition}

\section{Koch Surfaces and the Koch Crystals}\label{sec3}

\indent In this section we construct the family of fractal domains central to this paper and provide several main properties. But first, we recall the construction of the classical 2-dimensional Koch curve and modify it slightly to obtain an infinite family of related Koch $N$-curves ($N > 1$ odd). 

\subsection{Motivation}

Let $L$ be the compact segment in the $x$-axis of $\mathbb{R}^2$, centered at the origin with endpoints $(-1/2,0)$ and $(1/2,0)$. For $N > 1$, consider the following partitions of $L$  
$$L_N
:= \left\{ \left[-\frac{1}{2}, -\frac{1}{2} + \frac{1}{N} \right] \times \{0\}
%\left[-\frac{1}{2} + \frac{1}{N}, -\frac{1}{2} + \frac{2}{N}\right]
, \hdots , 
\left[\frac{1}{2} - \frac{1}{N}, \frac{1}{2}\right] \times \{0\} \right\}
=
\left\{\left[\frac{-1}{2} + \frac{i-1}{N}, \frac{-1}{2} + \frac{i}{N}\right] \times \{0\}\right\}_{i=1}^N,$$
consisting of $N$ compact intervals of length $1/N$ (see Figure \ref{fig:triangulations1}).

\begin{figure}[ht]
    \centering
    \includegraphics[scale=.9]{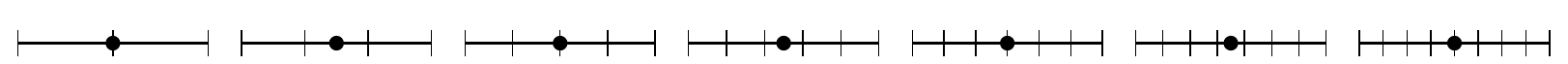}
    \caption{The point $(0,0)$ together with $L_2, \hdots, L_8$ respectively}
    \label{fig:triangulations1}
\end{figure}

Note that there does not exist a middle interval in $L_{2N}$ (for $N \geq 1$), that is, a unique interval containing the origin. With this in mind, we may use these $L_N$ to define the following family of fractals.

\begin{definition}
Let $N > 1$ such that $N \not\equiv 0 \,(\bmod\,\, 2)$. We define the \textbf{Koch $\mathbf{N}$-curve} to be the compact self-similar invariant set under $N+1$ mappings of ratio $1/N$. Out of these mappings, $N-1$ of them send $L$ to the interval $[\frac{-1}{2} + \frac{i-1}{N}, \frac{-1}{2} + \frac{i}{N}] \times \{0\}$ in $L_N$ for $i=1,\hdots,\frac{N-1}{2},\frac{N+1}{2},\hdots, N$. Notice we do not include a mapping which sends $L$ to the middle interval $[\frac{-1}{2N},\frac{1}{2N}] \times \{0\}$ in $L_N$. We do however include two additional mappings which send $L$ to the two compact intervals with endpoints $(\frac{-1}{2N},0)$, $(0,\frac{\sqrt{3}}{2N})$, and $(\frac{1}{2N},0)$, $(0,\frac{\sqrt{3}}{2N})$ respectively. Notice these two intervals together with the middle interval $[\frac{-1}{2N},\frac{1}{2N}] \times \{0\}$ form the edges of an equilateral triangle of side-length $\frac{1}{N}$ with vertices $(\frac{-1}{2N},0)$, $(0,\frac{\sqrt{3}}{2N})$, $(\frac{1}{2N},0)$.
\end{definition}

\begin{example}
The Koch 3-curve is the well-studied classical Koch curve, consisting of four self-similar copies of scale $1/3$. In Figure \ref{fig:koch3curveprefractals} (left), we present the images of $L$ under the four mappings which generate the Koch 3-curve in red, blue, purple, and yellow. In Figure \ref{fig:koch3curveprefractals} (right), we present the images of the left figure under the same four mappings in red, blue, purple, and yellow. Iterating this process we obtain a figure with four self-similar copies. 

\begin{figure}[ht]
    \centering
    \includegraphics{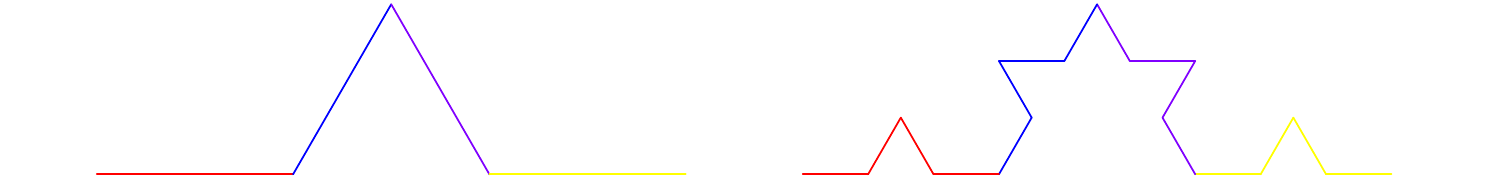}\\[10pt]
    \includegraphics{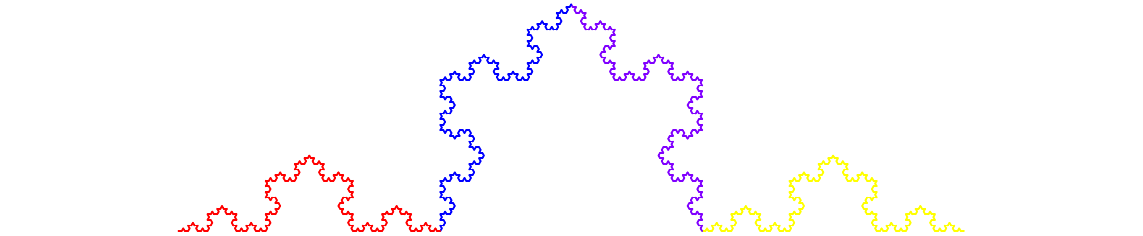}
    \caption{Koch 3-curve prefractals: first, second, and fourth iterations}
    \label{fig:koch3curveprefractals}
\end{figure}

% \begin{figure}[ht]
%     \centering
    
%     \caption{The Koch 3-curve}
%     \label{fig:koch3curve}
% \end{figure}
\end{example}

We contrast the classical Koch 3-curve with the following Koch 5-curve and 7-curve by presenting their first prefractals in Figure \ref{fig:koch5,7curves}.

\begin{figure}[ht]
    \centering
    \includegraphics[scale=1.5]{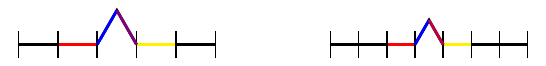}
    \caption{Koch 5-curve and 7-curve prefractals, first iteration (left and right resp.)}
    \label{fig:koch5,7curves}
\end{figure}

With this family of fractal curves, we may construct an associated family of fractal domains.

\begin{definition}
Let $N > 1$ such that $N \not\equiv 0 \,(\bmod\,\, 2)$. We define the \textbf{Koch N-snowflake} as the closed set enclosed by three congruent Koch $N$-curves, each pair of which intersect at precisely one point (see Figure \ref{fig:kochsnowflakes}).
\end{definition}

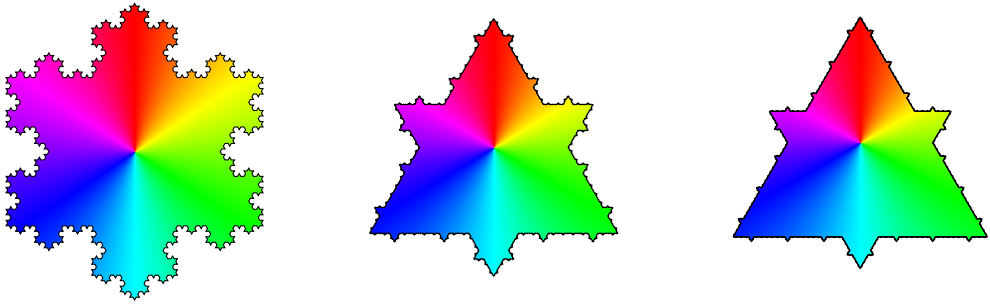
\begin{figure}[ht]
\centering
\begin{tikzpicture}
\shadedraw[shading=color wheel]
[l-system={rule set={F -> F-F++F-F}, step=1.2pt, angle=60,
   axiom=F++F++F, order=4}] lindenmayer system -- cycle;
\end{tikzpicture}
\hspace{30pt}
\raisebox{9pt}{
\begin{tikzpicture}
\shadedraw[shading=color wheel]
[l-system={rule set={F -> FF-F++F-FF}, step=0.15pt, angle=60,
   axiom=F++F++F, order=4}] lindenmayer system -- cycle;
\end{tikzpicture}
}
\hspace{30pt}
\raisebox{12pt}{
\begin{tikzpicture}
\shadedraw[shading=color wheel]
[l-system={rule set={F -> FFF-F++F-FFF}, step=0.04pt, angle=60,
   axiom=F++F++F, order=4}] lindenmayer system -- cycle;
\end{tikzpicture}
}
\caption{Koch 3-snowflake, 5-snowflake, and 7-snowflake prefractals (4th iteration)}
\label{fig:kochsnowflakes}
\end{figure}

In what follows, we will provide and study higher dimensional analogues of these constructions by replacing compact intervals (1-simplices) with triangles (2-simplices), and triangles with tetrahedrons (3-simplices).

\subsection{Construction}\label{sub3.1}

\indent Let $T$ be the compact region in the $xy$ plane of $\R^3$ enclosed by the equilateral triangle of side length 1 which is centered at the origin with vertices 
$$p_1 = \frac{\sqrt{3}}{6}\big(\cos(0),\sin(0), 0\big) \qquad p_2 = \frac{\sqrt{3}}{6}\big(\cos\left(2\pi/3\right),\sin\left(2\pi/3\right), 0\big) \qquad p_3 = \frac{\sqrt{3}}{6}\big(\cos\left(4\pi/3\right),\sin\left(4\pi/3\right),0\big).$$
Then for $N > 1$, we consider the following triangulations $T_N$ of $T$ consisting of $N^2$ equilateral triangles of scale $1/N$ (see Figure \ref{fig:triangulations2}).

\begin{figure}[ht]
    \centering
    \includegraphics[scale=0.40]{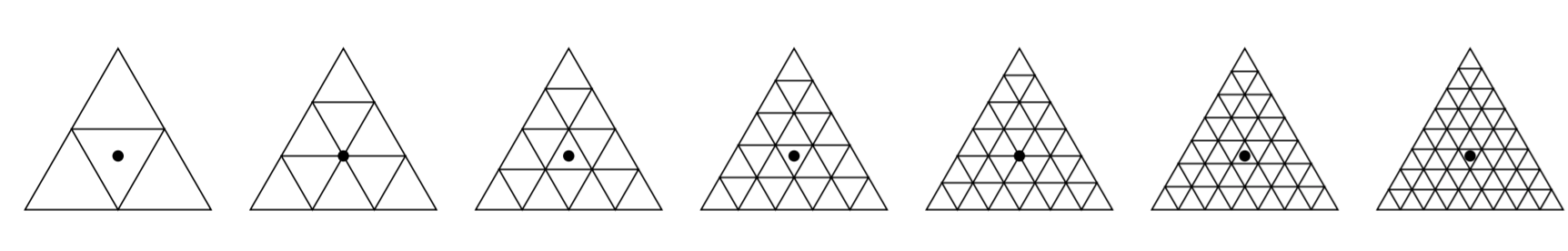}
    \caption{The point $(0,0,0)$ together with $T_2, \hdots, T_8$ respectively}
    \label{fig:triangulations2}
\end{figure}
\noindent Note that there does not exists a middle triangle in $T_{3N}$ (for $N \geq 1$), that is, a unique triangle containing the origin. With this in mind, we may use these $T_N$ to define the following family of fractals analogous to the construction of the Koch curve. 
\begin{definition}\label{kochNsurface}
Let $N > 1$ such that $N \not\equiv 0 \,(\bmod\,\, 3)$. We define the \textbf{Koch $\mathbf{N}$-surface} $\mathbf{K_N}$ to be the compact self-similar invariant set under the mappings $\F_N = \{F_{i,N}\}_{i=1}^{N^2+2}$ of ratio $1/N$ that send $T$ to each equilateral triangle except for the middle one in $T_N$, together with three additional mappings which send $T$ to the three equilateral triangles that form a regular tetrahedron with the removed middle triangle. By regular tetrahedron, we mean the boundary of a 3-dimensional simplex which is also a regular polytope. 

% We will index the three additional mappings by $F_{4,N},F_{5,N},F_{6,N}$ and choose these such that $p_1$ is mapped to the uppermost point of the aforementioned regular tetrahedron.
\end{definition}

Throughout this work, we reserve $N$ to play the role of determining both the scaling ratio and the number of mappings which generate the fractal $K_N$.

\begin{example}\label{koch2surface} The \textbf{Koch $\mathbf{2}$-surface} $\mathbf{K_2}$ is the compact self-similar invariant set under the family of mappings $\F_{2} = \{F_{i,2}\}_{i=1}^6$ given by
\begin{align*}
F_{1,2}(x,y,z) &= \left( \frac{x + \frac{\sqrt{3}}{3}}{2},\,\,\frac{y}{2},\,\, \frac{z}{2}\right)\\
F_{2,2}(x,y,z) &= \left(  \frac{x + \frac{\sqrt{3}}{3}\cos\left(\frac{2\pi}{3}\right)}{2},\,\, \frac{y + \frac{\sqrt{3}}{3}\sin\left(\frac{2\pi}{3}\right)}{2},\,\, \frac{z}{2}\right)\\
F_{3,2}(x,y,z) &= \left( \frac{x + \frac{\sqrt{3}}{3}\cos\left(\frac{4\pi}{3}\right)}{2},\,\, \frac{y + \frac{\sqrt{3}}{3}\sin\left(\frac{4\pi}{3}\right)}{2},\,\, \frac{z}{2}\right)\\
F_{4,2}(x,y,z) &= \left( -\frac{x}{6} + \frac{\sqrt{2}}{3}z + \frac{\sqrt{3}}{18},\,\, - \frac{y}{2},\,\, \frac{\sqrt{2}}{3}x + \frac{z}{6} + \frac{\sqrt{6}}{18}\right)\\
F_{5,2}(x,y,z) &= \left( \frac{x}{12} + \frac{\sqrt{3}}{4}y - \frac{\sqrt{2}}{6}z - \frac{\sqrt{3}}{36},\,\, -\frac{\sqrt{3}}{12}x + \frac{y}{4} + \frac{\sqrt{6}}{6}z + \frac{1}{12},\,\, \frac{\sqrt{2}}{3}x + \frac{z}{6} + \frac{\sqrt{6}}{18}\right)\\
F_{6,2}(x,y,z) &= \left(  \frac{x}{12} - \frac{\sqrt{3}}{4}y - \frac{\sqrt{2}}{6}z - \frac{\sqrt{3}}{36},\,\, \frac{\sqrt{3}}{12}x + \frac{y}{4} - \frac{\sqrt{6}}{6}z - \frac{1}{12},\,\, \frac{\sqrt{2}}{3}x + \frac{z}{6} + \frac{\sqrt{6}}{18}\right)\\
\end{align*}

In Figure \ref{fig:k2iterations} (left), we present the images of $T$ under the six mappings which generate the Koch 3-curve in red, blue, purple, yellow, green, and orange (the last of which is not visible). In Figure \ref{fig:k2iterations} (right), we present the images of the left figure under the same six mappings in red, blue, purple, yellow, green, and orange. Iterating this process we obtain Figure \ref{fig:k2}, resulting in six self-similar copies in red, blue, purple, yellow, green, and orange. 

\begin{figure}[ht]
    \includegraphics[scale=1]{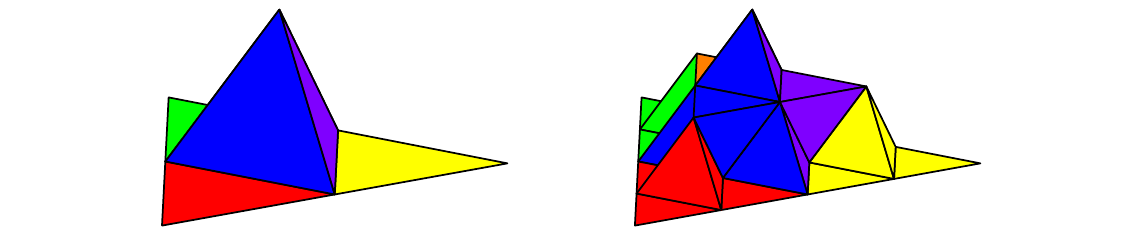}
    \caption{Koch 2-surface prefractals, first and second iterations}
    \label{fig:k2iterations}
\end{figure}

\begin{figure}[ht]
    \includegraphics[scale=1]{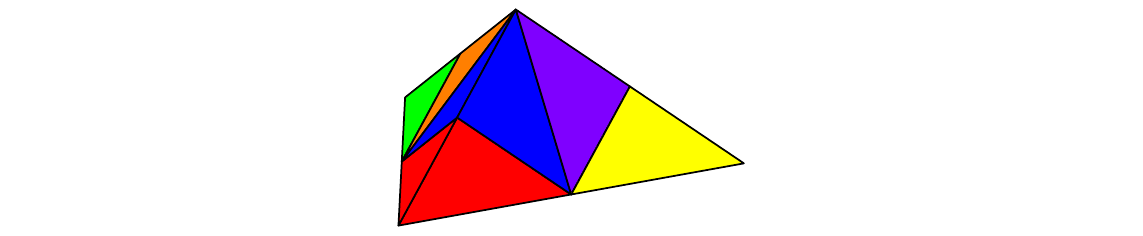}
    \caption{The Koch 2-surface $K_2$. From above, $K_2$ is indistinguishable from a tetrahedron, and its fractal features can only be viewed from below.}
    \label{fig:k2}
\end{figure}

 We will adopt the custom of writing $F_j(\cdot,\cdot,\cdot):= F_{j,2}(\cdot,\cdot,\cdot)$ ($j\in\{1,\ldots,6\}$) since it is a particularly difficult case, requiring closer examination.
\end{example}
Note that for $1 \leq j \leq 3$, $F_{j}$ contracts $T$ by a factor of $\frac{1}{2}$ and leaves $p_j$ fixed. Thus the maps $\{F_j\}_{i=1}^3$ generate Sierpi\'nski gaskets as seen in Figure \ref{fig:sierpinski}.
\begin{figure}[ht]
\includegraphics[scale=0.4]{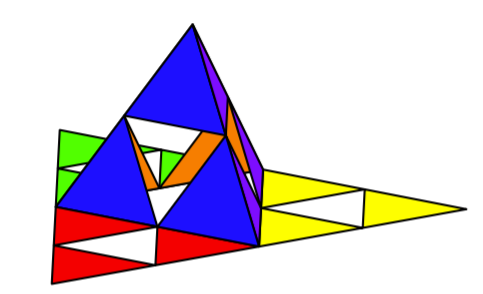}
\caption{Applying $F_{1}, F_{2}, F_{3}$ and then $\F_{2}$ to the triangle $T$.}
\label{fig:sierpinski}
\end{figure}

We contrast the Koch 2-surface with the following Koch 4-surface and 5-surface by presenting their first prefractals in Figure \ref{fig:kochsurf4n5}.

\begin{figure}[ht]
    \centering
    \includegraphics[scale=.35]{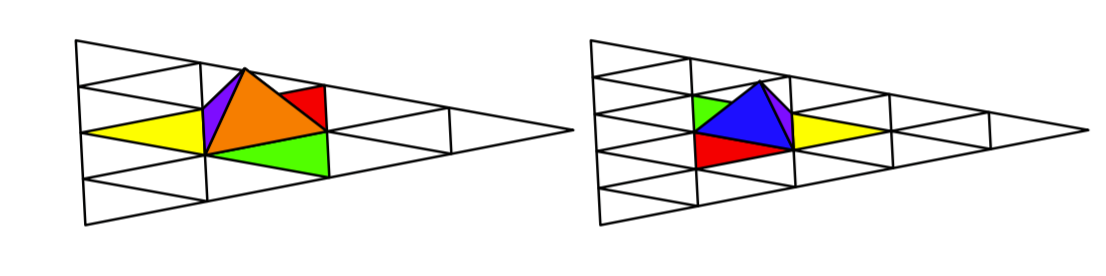}
    \caption{Koch 4-surface and 5-surface prefractals, first iteration (left and right resp.) We compare $K_4$ and $K_5$ with $K_2$ by coloring the six 1-cells around the ``peak'' of their first iterations in a similar fashion the coloring used in the first prefractal for $K_2$.}
    \label{fig:kochsurf4n5}
\end{figure}

We are now ready to define the fractals of main interest for this paper.

\begin{definition}\label{3.2} Let $N > 1$ such that $N \not\equiv 0 \,(\bmod\,\, 3)$. We define the \textbf{Koch $\mathbf{N}$-crystal} $\mathbf{\mathcal{C}_N}$ as the closed set enclosed by four congruent Koch $N$-surfaces, each pair of which intersect at precisely one edge. We then denote the boundary of $\mathcal{C}_{N}$ by $\partial\mathcal{C}_{N}$.
\end{definition}

\begin{example}
Note that, when we glue four Koch 2-surfaces as in Definition \ref{3.2}, we obtain a fractal whose outermost layer is the surface of a cube. In the spirit of Mandelbrot, we would like to remark how this figure resembles a geological geode with its smooth exterior containing a ``crystalline'' fractal interior. Thus, when considering the closed set enclosed by these Koch 2-surfaces, we see that the Koch 2-crystal $\mathcal{C}_2$ is a cube with side length $\sqrt{2}/2$. Because of this, $\mathcal{C}_2$ will play no role in Section \ref{sec7}, as it is not an interesting fractal domain. We will however study the Koch surface $K_2$ for its own sake in Sections \ref{sec3}, \ref{sec4}, and especially in Section \ref{sec5}. 
\begin{figure}[ht]
    \centering
    \includegraphics[scale=.4]{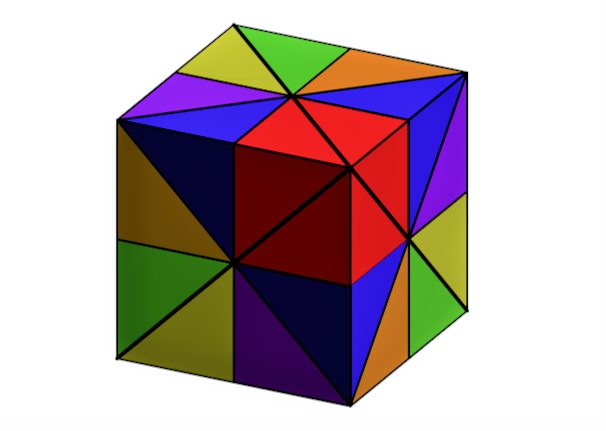}
    \caption{The Koch 2-crystal $\mathcal{C}_2$, which is indeed only a cube of side length $\frac{\sqrt{2}}{2}$.}
    \label{fig:C2}
\end{figure}
\end{example}

\begin{figure}[ht]
    \centering
    \includegraphics[scale=0.3]{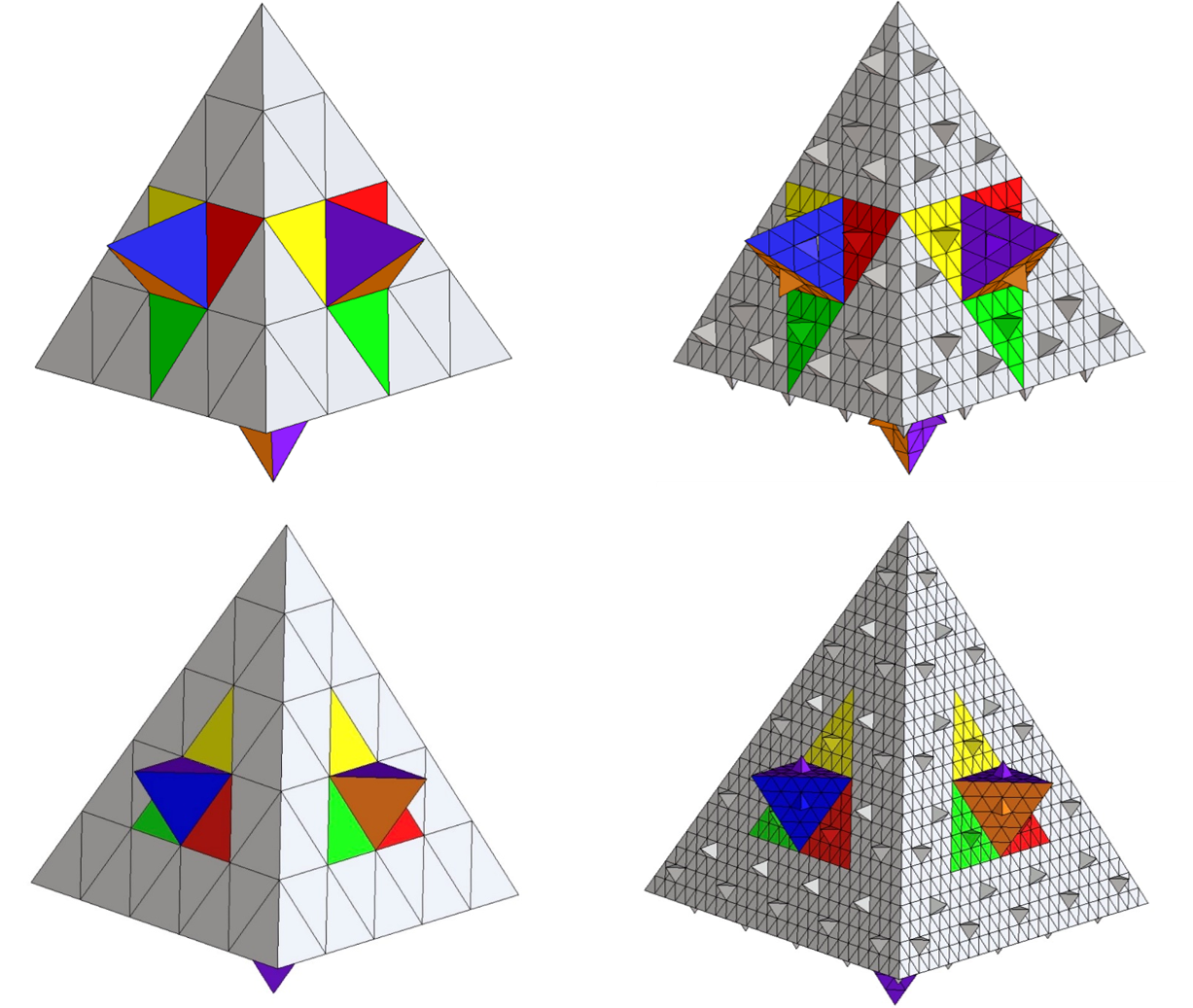}
    \caption{Koch 4-crystal, first and second iterations (top left and right resp.)\\
    \phantom{Figure \ref{fig:crystals}.\;\;\;\,} Koch 5-crystal, first and second iterations (bottom left and right resp.)}
    \label{fig:crystals}
\end{figure}
\subsection{Properties}\label{sub3.2}

\indent For $N \not\equiv 0$ (\text{mod} 3), let $K_N$ be the Koch $N$-surface generated by the iterated function system $\F_N$. One can see that each $\F_n$ satisfies the open set condition by considering the bounded open set enclosed by the tetrahedron with vertices $p_1, p_2, p_3,$ and the highest point $p_4 \in K_N$, i.e. $\pi_z(p_4) = \max\{z\mid (x,y,z) \in K_N\}$. Thus, by Theorem \ref{T1}, it follows that $s_N = \dim_{\hh}(K_N) = \log (N^2 + 2) / \log N$, which is the solution of the equation
$$ \sum_{k=1}^{N^2+2} \left(\frac{1}{N}\right)^{s_N} = (N^2 + 2)\left(\frac{1}{N}\right)^{s_N} = 1.$$
If $N>2$, then $\partial\mathcal{C}_N$ is the union of four copies of $K_N$, and thus $\dim_{\hh} (\partial\mathcal{C}_N) = \dim_{\hh}(K_N) = \log (N^2+2) / \log N$ due to the stability of the Hausdorff dimension. Furthermore, $\hh^{s_N}$ is an ${s_N}$-Ahlfors measure on $\partial\mathcal{C}_N$ for each $N\in \mathbb{N}\setminus\{1\}$ with $N \not\equiv 0$ (\text{mod} 3), where $s_N = \dim_{\hh}(\partial\mathcal{C}_N)$. Moreover, it is clearly seen that the interior of the set $\mathcal{C}_N\subseteq\mathbb{R\!}^{\,3}$ is an uniform domain.

In the case when $N=2$, we see from Figure \ref{fig:C2} that, while $K_2$ is a fractal of Hausdorff dimension $s_2 = \log 6 / \log 2$, the figure $\mathcal{C}_2$ is the cube of side-length 1 and $\partial \mathcal{C}_2$ is just its boundary of dimension $2$.

\section{Bounds for Hausdorff Measure}\label{sec4}

\indent In this section, we present the machinery needed in order to develop a process to compute sharp bounds for the Hausdorff measure of the Koch $N$-surfaces $K_N$ and $N$-crystals $\mathcal{C}_N$. The process will lead to an approximation tool to compute the Hausdorff measure of these fractal sets. Some key general results will be stated and proved. In the end, we will state the main results of the paper.\\
\indent We start with the following definition.

\begin{definition}\label{d4.1}
Let $K$ be the unique non-empty compact self-similar invariant set under an iterated function system (IFS) $\F = \{F_j\}_{j=1}^M$ satisfying the open set condition (OSC) where $F_j$ has ratio $0 < r_j < 1$. Let $\textbf{M} := \{1,2,\hdots, M\}$ and $n \geq 1$. We define the word space associated to $K$ as $\Omega := \textbf{M}^{\,\N}$ and $\Omega_n := \textbf{M}^n$ with the $n$-truncation map $[ \,\cdot\, ]_n: \Omega \rightarrow \Omega_n$ defined for a word $\omega = \omega_1 \omega_2 \cdots \in \Omega$ by $[\omega]_n := \omega_1 \cdots \omega_n$.
\end{definition}

We will not concern ourselves with the trivial case when $M = 1$. Notice that there is a relation between the word space $\Omega$ and the attractor $K$ of an IFS with $M$ maps, where we identify points in $K$ with infinite words, and regions with finite words. Namely for $\omega \in \Omega_n$, we define $K^{(\omega)} := F_{\omega}(K)$ where $F_{\omega_1 \omega_2 \cdots \omega_n}$ is given inductively as $F_{\omega_2 \cdots \omega_n} \circ F_{\omega_1}$. Moreover for $\omega \in \Omega$, we define the point $K^{(\omega)}$ as the unique point in $\bigcap_{n \in \N} K^{[\omega]_n}$. 
%We think of the words in $\Omega_n$ as addresses of cells in $K$ at the $n^{th}$ level, and the infinite words in $\Omega$ as the addresses of points (albeit points might have more than one address while regions will not). 
We will denote the natural probability measure on $K$ as $\mu$ where for $\omega = \omega_1 \cdots \omega_n \in \Omega_n$, we have that $\mu(K^{(\omega)}) = r^s_{\omega_1}\cdots r^s_{\omega_n}$. Since $\F$ satisfies the OSC, we also have that $\mu(K^{(\omega)}) = \sum_{j=0}^M \mu(K^{(\omega j)}) = \sum_{j=0}^M r_j^s \, \mu(K^{(\omega)})$.

\begin{definition}\label{d4.2}
Let $K^{n} := \{ K^{(\omega)} \mid \omega \in \Omega_n \}$ be the set of $n$-cells of $K$, where we reserve the notation $\Delta^{(n)}_i$ for elements of $K^{n}$, which we call $n$-cells. We also define $K^0 := K$.
\end{definition} 

We now present Proposition 1.1 in \cite{bao1}.

\begin{proposition}\label{p1}
For $n \geq 1$, $1 \leq k \leq M^n$, and $s = \dim_{\hh} K$, let 
$$b_k := \min_{\left\{\Delta^{(n)}_i\right\}_{i = 1}^k \subseteq K^n} \left\{ \frac{\left|\bigcup_{i = 1}^k \Delta^{(n)}_i\right|^s}{\mu\left(\bigcup_{i=1}^k \Delta_i^{(n)}\right)} \right\}$$
where the minimum is taken for all possible sets of $k$ elements of $K^n$, and let $a_n = \min_{1 \leq k \leq M^n} \{b_k\}$. If there exists a constant $\mathbf{a} > 0$ such that $a_n \geq\mathbf{a}$ for all $n$, then $\hh^s(K) \geq\mathbf{a}$. 
% We think of $|\bigcup_{i = 1}^k \Delta^{(n)}_i|^s$ as a loose analog of s-dimensional volume, and since $\mu$ is the natural probability measure on $K$, we can then think of $b_k$ as finding the collection of $k$ $n$-cells with the most ``density''. By means of this analogy, $a_n$ can be thought of as finding the collection of $n$-cells with the most ``density'' for every $n$.
\end{proposition}

\indent The sequence defined in Proposition \ref{p1} has a special consequence, as the following proposition taken from \cite{bao1} describes.

\begin{proposition}\label{p2}
For $n \geq 1$, the sequence $\{a_n\}$ defined in Proposition \ref{p1} is decreasing, with $\displaystyle\lim_{n \rightarrow \infty}a_n = \hh^s(K)$.
\end{proposition}

\indent One of the goals of this paper consists in finding a sequence of constants $\mathbf{a}$ (as in Proposition \ref{p1}) which increase towards $
\hh^s(K)$. This will be achieved using case-by-case analysis. To proceed, we add some additional definitions and notations.

\begin{definition} \label{case}
We say a proposition $P$ on the subsets of $K$ is a (valid) \textbf{case}, if 
\begin{itemize}
    \item For every $n$, there is a family $\Dset = \{\Delta_1^{(n)}, \hdots, \Delta_k^{(n)}\} \subseteq K^n$ such that $\Dcup$ satisfies $P$.
   \item For $n > 1$ and $\Dset \subseteq K^n$ such that $\Dcup$ satisfies case $P$, there is a family $\{\Delta^{(n-1)}_j\}_j \subseteq K^{n-1}$ such that $\Dcup \subseteq \bigcup_j \Delta^{(n-1)}_j$ and $\bigcup_j \Delta^{(n-1)}_j$ satisfies case $P$.
\end{itemize}
\end{definition}

\begin{remark}
Throughout the main proofs of the paper, the case $P$ will involve containment and intersection of $\Dcup$ with certain subsets of $K$. These are examples of cases in the sense of Definition \ref{case}. For example, Case 2b of Theorem 3 consists of $\Dcup$ intersecting exactly two of the base 1-cells $K^{(1)},K^{(2)},K^{(3)}$ in the Koch 2-surface $K$ while not being contained in the region $K^{(12)} \cup K^{(43)} \cup K^{(52)} \cup K^{(21)}$.
\end{remark}

\begin{definition}\label{d4.3}
In view of the notations in Definition \ref{d4.2}, we say that $\{\Delta_i^{(n)}\}_{i=1}^k \subset K^n$ is $P$-\textbf{scaleable} for a proposition $P$ if $\bigcup_{i=1}^k \Delta_i^{(n)} $ satisfies $P$ and there exists a similarity $S$ of ratio $r$ such that each $S^{-1}(\Delta_i^{(n)})$ is unique and in $K^{n-1}$, with $\bigcup_{i=1}^k S^{-1}(\Delta_i^{(n)})$ satisfying $P$. 
\end{definition}

We observe that if $\{\Delta_i^{(n)}\} \subset K^n$ is $P$-scaleable, then there exists an unique family in $K^n$ whose union coincides with that of $\{S^{-1}(\Delta_i^{(n)})\} 
\subset K^{n-1}$, thus satisfying $P$ as well. One then obtains the following result, which will be applied many times in the proofs of the central results of the paper to exclude certain subcases from consideration.

\begin{lemma}\label{scaleable}
Let 
$$a^{(P)}_n = \min_{1 \leq k \leq M^n} \min_{\Dset \subseteq K^n} \left\{ \frac{\left|\bigDcup\right|^s}{\mu\left(\bigDcup\right)} \,\Bigg\vert\, \bigDcup \text{ satisfies case } P\right\}$$
Furthermore, let
$$a'_n = \min_{1 \leq k \leq M^n} \min_{\Dset \subseteq K^n} \left\{ \frac{\left|\bigDcup \right|^s}{\mu\left(\bigDcup\right)} \,\Bigg\vert\, \bigDcup \ \text{satisfies case}\,\, P\,\,\text{and is not}\,\,P\text{-scaleable }\right\}.$$
Then $ a^{(P)}_n = a'_n$. That is, we may exclude $P$-scaleable families from consideration when calculating lower bounds for \,$\alpha^{(P)} := \displaystyle\lim_{n\rightarrow\infty} a^{(P)}_n$.
\end{lemma}

\begin{proof} Clearly $a_n^{(P)} \leq a'_n$. Suppose $\Dset \subset K^n$ is $P$-scaleable. Then $|S^{-1}(\Dcup)| = r^{-1}|\Dcup|$. Now note that every $S^{-1}(\Delta_i^{(n)}) \in K^{n-1}$ is the union of $\{\Delta_{i_j}^{(n)}\}_{j=1}^{M} \subset K^n$. By considering the family $\bigcup_{i=1}^k \{\Delta_{i_j}^{(n)}\}_{j=1}^M \subset K^n$, this satisfies case $P$ since $\bigcup_{i=1}^k \bigcup_{j=1}^M \Delta_{i_j} = \bigcup_{i=1}^k S^{-1}(\Delta_i^{(n)})$, and $Mk \leq M^n$ with
$$\frac{\left|\bigDcup\right|^s}{\mu\left(\bigDcup\right)} = \frac{r^{s} \left|S^{-1}\left(\bigDcup\right)\right|^s}{\mu\left(\bigDcup\right)} = \frac{\left|\bigcup_{i,j}\Delta_{i_j}^{(n)}\right|^s}{ \sum_{i=1}^k r^{-s} \mu\left(\Delta_i^{(n)}\right)} = \frac{\left|\bigcup_{i,j} \Delta_{i_j}^{(n)}\right|^s}{\mu\left(\bigcup_{i,j} \Delta_{i_j}^{(n)}\right)},$$
since $\sum_{i=1}^k r^{-s} \mu(\Delta_i^{(n)}) = \sum_{i=1}^k \mu(S^{-1}(\Delta_i^{(n)})) = \sum_{i=1}^k \mu( \bigcup_{j=1}^M \Delta_{i_j}^{(n)}) = \mu(\bigcup_{i,j} \Delta_{i_j}^{(n)})$. We may repeat this process if the family $\bigcup_{i=1}^K \{ \Delta_{i_j}^{(n)}\}_{j=1}^M$ is $P$-scaleable and so forth. We must eventually obtain a family that is not $P$-scaleable. Indeed, if one were able to apply this process $n$ times, then $M^nk \leq M^n$ and $k = 1$. Thus, the family obtained after $n$ steps must be exactly $K^n$, which is not $P$-scaleable due to the uniqueness condition $\{S^{-1}(\Delta^{(n)}_i)\}$ would need to satisfy. Thus the value $|\Dcup| \,/\, \mu(\Dcup)$ must be larger than the value achieved by some family which is not $P$-scaleable. Therefore $a_n^{(P)} \geq a'_n$. 
\end{proof}

\indent The following key result will be constantly applied in the proof of the central results of the paper, and has value of its own as it can be applied to general fractals. We present the general version below, which allows us to bound the limit $\alpha^{(P)} = \lim a_n^{(P)}$ for a case $P$ by the sequence $a_n$ multiplied by a factor which is: 
\begin{itemize}
    \item proportional to the diameter of $K$ and to any lower bound $\beta$ on the diameters of families $\Dcup$ which satisfy case $P$; and
    \item inversely proportional to the largest scaling ratio of maps in the IFS generating $K$, and the maximum Hausdorff distance between the 1-cells of $K$.
\end{itemize} 

We will then provide a more specific version of this result as Corollary \ref{T3corPar}, which will be useful to us when $K$ is a Koch $N$-surface $K_N$. Finally, we will find such lower bounds $\beta$ to obtain lower bounds on the Hausdorff dimension of $K_N$ by virtue of this following theorem.

\begin{theorem}\label{T3}
Let $\{a_n\},\,\{a_n^{(P)}\}$ be as in Proposition \ref{p1} and Lemma \ref{scaleable}, respectively, and let $\beta \leq |\Dcup|$, for every $\Dset \in K^n$ such that $ \Dcup \text{ satisfies } P$. Then 
$$\alpha^{(P)} \geq a_n |K|^{s} \exp{\frac{-s \gamma_{n}}{\beta (1 - r_{max})}},$$
\noindent where \,$r_{max} := \displaystyle\max_{1\leq i\leq M} r_i$
% , \,$r_{\min} := \displaystyle\min_{1\leq i\leq M} r_i$, 
and \,$\gamma_n := 2r^{n}_{\max} \;\displaystyle \max_{1 \leq \ell, k \leq M} d_H(F_\ell(K),F_k(K))$.\\
\end{theorem}

\begin{proof}
    Since $\hh^s(K) = \hh^s(K/|K|)/|K|^s$, we may suppose that $|K| = 1$.
    We now construct a proof motivated by the procedure found in \cite{bao2}. Indeed, since $P$ is a case for $n > 1$ and the family of $n$-cells $\{\Delta^{(n)}_i\}_i$, there exists a collection of $(n-1)$-cells $\Delta^{(n-1)}_j \in K^{n-1}$ such that $\Delta^{(n)}_i \subset \Delta^{(n-1)}_j$, $\bigcup_j \Delta^{(n-1)}_j$ satisfies case $P$, and $\Delta^{(n-1)}_1, \hdots, \Delta^{(n-1)}_{k_{n-1}}$ are all taken to be distinct. %due to the open set condition. %One can see that $k \leq M\,k_{n-1}$ since M $n$-cells are contained in a $(n-1)$-cell. 
    Next, we claim that
    \begin{equation}\label{T3.1}
    \left|\bigcup_{j=1}^{k_{n-1}} \Delta_j^{(n-1)}\right|\, \leq\, d + \gamma_{n-1},\indent\,\,\,\textrm{ for}\,\,\,d : = \left|\Dcup \right|.
    \end{equation}
    To establish the claim, we proceed as follows. Note that for $x,y \in \bigcup_{j} \Delta^{(n-1)}_j$, $x \in \Delta^{(n-1)}_x$ and $y \in \Delta^{(n-1)}_y$ for some $\Delta^{(n-1)}_x,\Delta^{(n-1)}_y \in \{\Delta^{(n-1)}_j\}_j$. We then obtain 
    $$\|x-y\| \leq \min_{a \,\in\, \bigcup_i \Delta_i^{(n)} \cap\, \Delta_x^{(n-1)}}\|x-a\| + d + \min_{b \,\in\, \bigcup_i \Delta_i^{(n)} \cap\, \Delta_y^{(n-1)}}\|b-y\|.$$
Here we write $\bigcup_i \Delta_i^{(n)} \cap \Delta_x^{(n-1)}$ for simplicity, where one should take the union of the $\Delta_i^{(n)}$ which are contained in $\Delta_x^{(n-1)}$. Taking the supremum over $x,y \in \bigcup_j \Delta_j^{(n-1)}$ this yields
\begin{align*}
\left|\bigcup_j \Delta^{(n-1)}_j\right| 
&\leq d + 2 \max_{x \in \bigcup_j \Delta_j^{(n-1)}} \quad \min_{a \in \bigcup_i \Delta_i^{(n)} \cap \Delta_x^{(n-1)}} \|x-a\|\\
&= d + 2 \max_j \max_{x \in \Delta_j^{(n-1)}} \quad \min_{a \in \bigcup_i \Delta_i^{(n)} \cap \Delta_j^{(n-1)}} \|x-a\|.
\end{align*}
Re-scaling each $\Delta_j^{(n-1)}$ onto $K$, every $\Delta_i^{(n)}$ contained in $\Delta_j^{(n-1)}$ is mapped to some $\Delta_k^{(1)}$. Taking the maximum over all families of 1-cells, one bounds the previous term as follows
\begin{align*}
\left|\bigcup_j \Delta^{(n-1)}_j\right| 
&\leq d + 2 r_{max}^{n-1} \max_{x \in K} \;\max_{\{\Delta^{(1)}_k\}_k \subseteq K^1} \; \min_{a \in \bigcup_k \Delta^{(1)}_k}\|x-a\|\\
&\leq
d + 2 r_{max}^{n-1} \max_{x \in K} \;\max_{\{\Delta^{(1)}_k\}_k \subseteq K^1} \; d_H\left(\{x\},\bigcup_k \Delta_k^{(1)}\right)\\
&\leq
d + 2 r_{max}^{n-1} \max_{x \in K} \;\max_{\{\Delta^{(1)}_k\}_k \subseteq K^1} \; \max_{\Delta_{k_0} \in \{\Delta^{(1)}_k\}_k} d_H\left(\{x\},\Delta_{k_0}^{(1)}\right)\\
&=
d + 2 r_{max}^{n-1} \max_{x \in K} \;\max_{\Delta^{(1)}_k \in K^1} \; d_H\left(\{x\},\Delta_{k}^{(1)}\right).
\end{align*}
Recall that each $\Delta_k^{(1)} \in K^1$ is given by $F_k(K)$, hence
$$
\left|\bigcup_j \Delta^{(n-1)}_j\right| \leq
d + 2r^{n-1}_{max} \max_{1 \leq k \leq M} d_H(K,F_k(K)),
$$
%where $K^{\min} \in K^1$ is some 1-cell with minimal diameter $r_{\min}$. 
where the latter value is precisely $d + \gamma_{n-1}$. Thus (\ref{T3.1}) is established, as desired. From here, taking into account the monotonicity of the function $f(x) = (x+\gamma_{n-1})/x$ for $x > 0$ and a fixed $n$, we deduce that\, 
    $|\bigcup_{i} \Delta_i^{(n-1)}|/d \leq (d + \gamma_{n-1})/d \leq (\beta + \gamma_{n-1})/\beta$. Moverover, the fact that $\cup_i \Delta_i^{(n)} \subseteq \cup_j \Delta_j^{(n-1)}$ implies $\mu(\cup_i \Delta_i^{(n)}) \leq \mu(\cup_j \Delta_j^{(n-1)})$. We then obtain
    \begin{equation}\label{T3.2}
    \frac{d^s}{\mu\left(\bigcup_i \Delta_i^{(n)}\right)} \geq \left(\frac{\beta}{\beta + \gamma_{n-1}}\right)^s \frac{\left|\bigcup_i \Delta_i^{(n-1)}\right|^s}{\mu\left(\bigcup_j \Delta^{(n-1)}_j\right)}.
    \end{equation}
 Taking infima over both sides in (\ref{T3.2}), we conclude
    $$a_n^{(P)}
    % \geq \left(\frac{\beta}{\beta + \gamma_{n-1}}\right)^s \frac{\left|\bigcup_i \Delta_i^{n-1}\right|^s}{\mu\left(\bigcup_j \Delta^{(n-1)}_j\right)}
    \geq \left(\frac{\beta}{\beta + \gamma_{n-1}}\right)^s a_{n-1}^{(P)}.$$
    Then, for any $m \geq 1$, proceeding inductively, we arrive at
    \begin{equation}\label{T3.3}
    a_{n+m}^{(P)} \geq a_n^{(P)} \prod_{i = n}^{m+1} \left(\frac{\beta}{\beta + \gamma_i}\right)^s = a_n^{(P)} \prod_{i = n}^{m+1} \left(1 + \frac{\gamma_i}{\beta}\right)^{-s}.
    \end{equation}
    Taking logarithms on both sides in (\ref{T3.3}), and using the inequality $\ln(1+x) < x$, valid for $x >0$, we find that
    \begin{equation}\label{T3.4}
        \ln(a_{n+m}^{(P)}) \geq \ln(a_n^{(P)}) - s \sum_{i=n}^{m+1} \ln\left(1 + \frac{\gamma_i}{\beta} \right) \geq \ln(a_n^{(P)}) - s\sum_{i=n}^{m+1} \frac{\gamma_i}{\beta}.
    \end{equation}
    Proceeding as in Propositions \ref{p1} and \ref{p2}, one sees that the sequence $\{a_n^{(P)}\}$ is decreasing and bounded. Setting $\alpha^{(P)} := \displaystyle\lim_{n \rightarrow \infty} a_n^{(P)}$, and letting $m \rightarrow \infty$ in (\ref{T3.4}), we have
    $$\ln(\alpha^{(P)}) \geq \ln(a_n^{(P)}) -s \left[ \frac{\gamma_{n} / \beta}{1 - r_{max}}\right] = \ln\left(a_n^{(P)} \exp{ \frac{-s \gamma_{n}}{\beta (1 - r_{max})}}\right).$$
    Therefore, $\alpha^{(P)} \geq a_n^{(P)}\exp{\frac{-s \gamma_{n}}{\beta (1 - r_{max})}} \geq a_n \exp{\frac{-s \gamma_{n}}{\beta (1 - r_{max})}}$, completing the proof.
\end{proof}
\indent\\
\indent An useful form of the preceding theorem for particular types of sets $K$ reads as follows.

\begin{corollary}\label{T3corPar}
Under the assumptions and notations of Theorem \ref{T3}, assume that $|K| = 1$, $r_{max} = r_{\min} = r$, and $\max_{\ell,k} d_H(F_\ell(K),F_k(K)) = 1 - r$. Then
$$\alpha_n^{(P)} \geq a_n \exp{-\frac{s \cdot 2r^{n}}{\beta}}.$$
\end{corollary}

\indent It is easily verified that fractals such as the Cantor set, Sierpinski gasket, Koch curve, and Koch $N$-surfaces all satisfy the conditions in Corollary \ref{T3corPar}.

\indent\\
\indent We now present the central results of the paper.

\begin{theorem}\label{T2}
Let $a_n$ be a sequence given as in Proposition \ref{p1} for the Koch 2-surface $K_2$ as seen in Example \ref{koch2surface}. Then for every $n\in\mathbb{N\!}\,$, the Hausdorff measure of $K_2$ satisfies the following estimation:
\begin{equation}\label{T2est}
a_n \,\geq \,\hh^{s_2}(K_2) \,\geq \,a_n\, \exp{-\frac{{s_2}(\sqrt{2}+\sqrt{6})}{2^{n-6}}},
\end{equation}
where we recall that $s_2:=\frac{\log(6)}{\log(2)}$.
\end{theorem}
\indent\\
\indent One can calculate and find that $a_1= b_3 = 2\left|K^{(4)} \cup K^{(5)} \cup K^{(6)}\right|^{s_2} = 2\left(\frac{\sqrt{6}}{4}\right)^{s_2}$.
% , which together with Theorem \ref{T2} implies that
% $$8.90 \times 10^{-140} \approx 2\left(\frac{\sqrt{6}}{4}\right)^{s_2} \exp{-2^5(\sqrt{2}+\sqrt{6}) \, s_2}\,\leq\,\hh^{s_2}(K_2)\,\leq\, 2\left(\frac{\sqrt{6}}{4}\right)^{s_2} \approx 0.5629546.$$
% \indent\\

\begin{theorem}\label{generalmainthm}
Let $N > 2$ such that $N \not\equiv 0\, (\text{mod } 3)$, and let $a_n$ be a sequence as in Proposition \ref{p1} for the Koch $N$-surface $K_N$ defined by Definition \ref{kochNsurface}. Then for every $n \in \N$, the Hausdorff measure of the surface $K_N$  satisfies the following estimation:

$$a_n \,\geq\, \hh^{s_N}(K_N) \,\geq\, a_n \exp{-\frac{s_N \sqrt{6}}{N^{n-3}}}$$
where we recall that $s_N := \frac{\log(N^2+2)}{\log(N)}$.
\end{theorem}

\section{Proof of Theorem \ref{T2}}\label{sec5}
The leftmost inequality follows immediately from Proposition \ref{p2}. We now focus on the remaining inequality. To aid in legibility, we shall write $K := K_2$ for the Koch 2-surface. Let $\{ \Delta_i^{(n)} \}_i$ be a collection of $n$-cells in $K$ and let $d = |\bigcup_{i=1}^k \Delta_i^{(n)}|$ be the diameter of such collection. We will organize the proof by cases, based on how many of the bottom 1-cells $K^{(1)},K^{(2)},K^{(3)}$ the family $\bigcup_{i = 1}^k \Delta^{(n)}_i$ intersects. These cases will subdivided by how many of the top 1-cells $K^{(4)},K^{(5)},K^{(6)}$ the family $\bigcup_{i = 1}^k \Delta^{(n)}_i$ intersects and strategically excluding scenarios when $\bigcup_{i = 1}^k \Delta^{(n)}_i$ is scaleable. This will allow us find a constant $\beta > 0$ as in Theorem \ref{T3} and Corollary \ref{T3corPar}; that is, $\beta \leq d$ whenever $\{\Delta_i^{(n)}\}$ satisfies the case under consideration. 

\begin{remark}
In order to find such a lower bound $\beta > 0$, we will consider families $\Dset$ such that $k$ is as small as possible (so they have minimal diameter for a fixed $n$) while satisfying the case in question. Intuitively, as $n \to \infty$, such $\Dset$ approximate a family of points in $K$. The diameter of this set of points is meant to yield an optimal value for $\beta > 0$ independent of $n$. Moreover, since our cases involve intersecting certain regions in $K$, each of these points represents a ``constraint'' corresponding to one of these regions. Throughout our proof, we provide diagrams depicting these sets of points (or constraints) whose diameter is $\beta$.
\end{remark}

We first provide a definition and lemma will be used for cases where there is a square ``critical region'', that is, a square which intersects families $\Dset \subset K^n$ satisfying the case under consideration, with arbitrarily small diameters as we let $n \to \infty$. As these regions cause problems when trying to find lower bounds on diameters, we must build some machinery to tackle these obstacles throughout the proof of our first main result.

\begin{definition}
Let $A_\ell \subset K$ be a square of side length $\frac{\ell\sqrt{2}}{2}$ with a distinguished diagonal $L$ of length $\ell$. Furthermore, let $A_\ell^n := \{ \Delta \cap A_\ell \mid \Delta \in K^n\}$ be the set of $n$-cells of $A_l$. We say that $\{\Theta^{(n)}_i\}_{i=1}^k \subseteq A_\ell^n$ is {\bf P-scaleable in A$_\ell$}, if there exists a $P$-scaleable $\{\Delta_i\}_{i=1}^k \subset K^n$ such that $\Delta_i^{(n)} \cap A_\ell = \Theta^{(n)}_i$ for all $1 \leq i \leq k$.  

\begin{figure}[ht]
\begin{tikzpicture}[scale=1.3]
\filldraw[fill=black!20] (0,1) -- (0,-1) -- (-1,0) -- (0,1);
\filldraw[fill=black!60] (0,1) -- (0,-1) -- (1,0) -- (0,1);
\draw (-0.5,0.5) -- (0.5,-0.5);
\draw (0.5,0.5) -- (-0.5,-0.5);
\draw (-0.75,0.25) -- (0.25,-0.75);
\draw (-0.25,0.75) -- (0.75,-0.25);
\draw (-0.25,-0.75) -- (0.75,0.25);
\draw (-0.75,-0.25) -- (0.25,0.75);
\draw[blue, ultra thick] (0,1) -- (0,-1);
\node at (0,1.2) {\textcolor{blue}{L}};
\end{tikzpicture}
\caption{The square $A_\ell$ with the diagonal $L$ down the middle in blue.}
\end{figure}
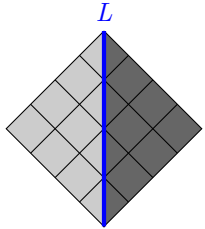
\end{definition}

Armed with the previous definition, we may express the required conditions for the following lemma, which allows us to obtain lower bounds on diameters of families $\Dset$ satisfying
Cases 3c.ii, 3c.iii, and 4b. As mentioned before, we will provide diagrams depicting sets of constraints whose diameter is $\beta$, and we will color each constraint by the region it represents. This is meant to help the reader get their bearing on the process of finding these values for $\beta$. Later in the proof of Theorem 3, we simply color all dots with green. We will also graph the segments between these points (in green) in order to aid the reader in computing the diameters of these collections of points.

\begin{lemma} \label{square}
Let $\Dset \subset K^n$ and $\Theta^{(n)}_i := \Delta_i^{(n)} \cap A_\ell \neq \varnothing$ for all $1 \leq i \leq k$. Suppose that the following conditions hold
\begin{itemize}
    \item $|\Dcup| \leq |\bigcup_{i=1}^k \Theta^{(n)}_i|$,
    \item $\bigcup_{i=1}^k \Theta^{(n)}_i$ intersects both $H_L$ and $H_R$, the regions left and right of the distinguished diagonal $L$ respectively,
    \item $\Dset$ is $P$-scaleable if and only if $\{\Theta^{(n)}_i\}_{i=1}^k$ is $P$-scaleable in $A_\ell$.
\end{itemize}
Then $\Dset$ is scaleable or 
$$d := |\Dcup| \geq \ell (\sqrt{3} - 1) \left(\frac{1}{2}\right)^4.$$
\end{lemma}

\begin{proof}
We define $R_1$, $R_2$, $R_3$ to be the following self-similar regions of scale $1/2$. 

\begin{figure}[ht]
\begin{tikzpicture}[scale=1.3]
\filldraw[fill=black!20] (0,1) -- (0,-1) -- (-1,0) -- (0,1);
\filldraw[fill=black!60] (0,1) -- (0,-1) -- (1,0) -- (0,1);
\draw (-0.5,0.5) -- (0.5,-0.5);
\draw (0.5,0.5) -- (-0.5,-0.5);
\draw (-0.75,0.25) -- (0.25,-0.75);
\draw (-0.25,0.75) -- (0.75,-0.25);
\draw (-0.25,-0.75) -- (0.75,0.25);
\draw (-0.75,-0.25) -- (0.25,0.75);
\draw[ultra thick] (0,1) -- (0,-1);
%node at (0,1.2) {\textcolor{blue}{L}};
\filldraw[fill=red, fill opacity = 0.4, draw=red, ultra thick] (0,0) -- (0.5,0.5) -- (0,1) -- (-0.5,0.5) -- (0,0);
\filldraw[fill=blue, fill opacity = 0.4, draw=blue, ultra thick] (0,0-1) -- (0.5,0.5-1) -- (0,1-1) -- (-0.5,0.5-1) -- (0,0-1);
\filldraw[fill=violet, fill opacity = 0.4, draw=violet, ultra thick] (0,0-0.5) -- (0.5,0.5-0.5) -- (0,1-0.5) -- (-0.5,0.5-0.5) -- (0,0-0.5);
\node at (1.25,0) {\textcolor{white}{$R_2$}};
\node at (-1.25,0) {\textcolor{violet}{$R_2$}};
\node at (-1.25,0.5) {\textcolor{red}{$R_1$}};
\node at (-1.25,-0.5) {\textcolor{blue}{$R_3$}};
\end{tikzpicture}
\caption{The square $A_\ell$ with the regions $R_1$ (red), $R_2$ (purple), and $R_3$ (blue) from top to bottom respectively}
\end{figure}
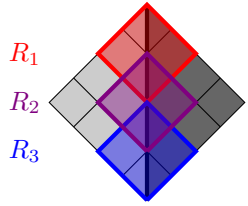

We provide bounds by cases depending on how many of the regions $R_i$ the figure $\bigcup_i \Theta^{(n)}_i$ is contained in.
\begin{itemize}
    \item When $\bigcup_i \Theta^{(n)}_i$ lies in $R_1$,\, $R_2$, or $R_3$, it follows that $\bigcup_i \Theta^{(n)}_i$ is $P$-scaleable in $A_\ell$ and can be excluded from consideration by Lemma \ref{scaleable}. 
    \item When $\bigcup_i \Theta^{(n)}_i$ belongs to either $R_1 \cup R_2$, or $R_2 \cup R_3$, and none of the previous cases, we have that $\bigcup_i \Theta^{(n)}_i \not\subset R_1 \cap R_2$ and $\bigcup_i \Theta^{(n)}_i \not\subset R_2 \cap R_3$. By symmetry, we can assume $\bigcup_i \Theta^{(n)}_i \subset R_1 \cup R_2$ and $\bigcup_i \Theta^{(n)}_i \not\subset R_1 \cap R_2$. Thus, $\bigcup_i \Theta^{(n)}_i$ must intersect $R_1\setminus R_2$,\, $R_2\setminus R_1$,\, $H_L$, and $H_R$. A quick calculation shows that $d \geq \ell (\sqrt{3} - 1) \left(\frac{1}{2}\right)^4$
\end{itemize}
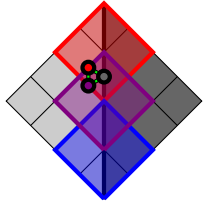
\begin{figure}[ht]
\begin{tikzpicture}[scale=1.3]
\filldraw[fill=black!20] (0,1) -- (0,-1) -- (-1,0) -- (0,1);
\filldraw[fill=black!60] (0,1) -- (0,-1) -- (1,0) -- (0,1);
\draw (-0.5,0.5) -- (0.5,-0.5);
\draw (0.5,0.5) -- (-0.5,-0.5);
\draw (-0.75,0.25) -- (0.25,-0.75);
\draw (-0.25,0.75) -- (0.75,-0.25);
\draw (-0.25,-0.75) -- (0.75,0.25);
\draw (-0.75,-0.25) -- (0.25,0.75);
\draw[ultra thick] (0,1) -- (0,-1);
%node at (0,1.2) {\textcolor{blue}{L}};
\filldraw[fill=red, fill opacity = 0.4, draw=red, ultra thick] (0,0) -- (0.5,0.5) -- (0,1) -- (-0.5,0.5) -- (0,0);
\filldraw[fill=blue, fill opacity = 0.4, draw=blue, ultra thick] (0,0-1) -- (0.5,0.5-1) -- (0,1-1) -- (-0.5,0.5-1) -- (0,0-1);
\filldraw[fill=violet, fill opacity = 0.4, draw=violet, ultra thick] (0,0-0.5) -- (0.5,0.5-0.5) -- (0,1-0.5) -- (-0.5,0.5-0.5) -- (0,0-0.5);
\filldraw[draw=black, ultra thick] (-0.16,0.342) -- (-0.16,0.158);
\filldraw[draw=green, thick] (-0.16,0.342) -- (-0.16,0.158);
\filldraw[draw=black, ultra thick] (0,0.25) -- (-0.16,0.158);
\filldraw[draw=green, thick] (0,0.25) -- (-0.16,0.158);
\filldraw[draw=black, ultra thick] (-0.16,0.342) -- (0,0.25);
\filldraw[draw=green, thick] (-0.16,0.342) -- (0,0.25);
\filldraw[fill=black!60 , draw=black , ultra thick] (0,0.25) circle (1.7pt);
\filldraw[fill=red , draw=black , ultra thick] (-0.16,0.342) circle (1.7pt);
\filldraw[fill=violet , draw=black , ultra thick] (-0.16,0.158) circle (1.7pt);
\end{tikzpicture}
\caption{The constraints from $R_1 \setminus R_2$ (red, uppermost), $R_2 \setminus R_1$ (purple, lowermost), and $H_R$ (dark gray, rightmost) form an equilateral triangle with side length $\ell (\sqrt{3} - 1)(\frac{1}{2})^4$. We note that there exists a mirrored configuration of constraints by symmetry, which includes a constraint from $H_L$ instead of $H_R$.}
\end{figure}

\begin{itemize}
    \item When $\bigcup_i \Theta^{(n)}_i \subset R_1 \cup R_2 \cup R_3$ and none of the previous cases, one sees that $\bigcup_i \Theta^{(n)}_i$ must intersect $R_1\setminus R_2$ and $R_3\setminus R_2$. Then $d \geq \ell \left(\frac{1}{2}\right)^2$.
\end{itemize}
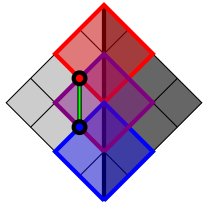
\begin{figure}[ht]
\begin{tikzpicture}[scale=1.3]
\filldraw[fill=black!20] (0,1) -- (0,-1) -- (-1,0) -- (0,1);
\filldraw[fill=black!60] (0,1) -- (0,-1) -- (1,0) -- (0,1);
\draw (-0.5,0.5) -- (0.5,-0.5);
\draw (0.5,0.5) -- (-0.5,-0.5);
\draw (-0.75,0.25) -- (0.25,-0.75);
\draw (-0.25,0.75) -- (0.75,-0.25);
\draw (-0.25,-0.75) -- (0.75,0.25);
\draw (-0.75,-0.25) -- (0.25,0.75);
\draw[ultra thick] (0,1) -- (0,-1);
%node at (0,1.2) {\textcolor{blue}{L}};
\filldraw[fill=red, fill opacity = 0.4, draw=red, ultra thick] (0,0) -- (0.5,0.5) -- (0,1) -- (-0.5,0.5) -- (0,0);
\filldraw[fill=blue, fill opacity = 0.4, draw=blue, ultra thick] (0,0-1) -- (0.5,0.5-1) -- (0,1-1) -- (-0.5,0.5-1) -- (0,0-1);
\filldraw[fill=violet, fill opacity = 0.4, draw=violet, ultra thick] (0,0-0.5) -- (0.5,0.5-0.5) -- (0,1-0.5) -- (-0.5,0.5-0.5) -- (0,0-0.5);
\filldraw[draw=black, ultra thick] (-1/4,1/4) -- (-1/4,-1/4);
\filldraw[draw=green, thick] (-1/4,1/4) -- (-1/4,-1/4);
\filldraw[fill=red , draw=black , ultra thick] (-1/4,1/4) circle (1.7pt);
\filldraw[fill=blue , draw=black , ultra thick] (-1/4,-1/4) circle (1.7pt);
\end{tikzpicture}
\caption{The constraints from $R_1 \setminus R_2$ (red, uppermost) and $R_3 \setminus R_2$ (blue, lowermost), which are at distance $\ell (\frac{1}{2})^2$.  We note that there exists a mirrored configuration of constraints by symmetry, corresponding the families contained in $H_R$ instead of $H_L$.}
\end{figure}
\begin{itemize}
    \item When none of the previous cases are satisfied, we have that $\bigcup_i \Theta^{(n)}_i \not\subset R_1 \cup R_2 \cup R_3$. By symmetry we may assume $\bigcup_i \Theta^{(n)}_i$ intersects $H_L\setminus\left(\bigcup_i R_i\right)$ and $H_R$. Then $d \geq \ell \left(\frac{1}{2}\right)^3$.
\end{itemize}
\begin{figure}[ht]
\begin{tikzpicture}[scale=1.3]
\filldraw[fill=black!20] (0,1) -- (0,-1) -- (-1,0) -- (0,1);
\filldraw[fill=black!60] (0,1) -- (0,-1) -- (1,0) -- (0,1);
\draw (-0.5,0.5) -- (0.5,-0.5);
\draw (0.5,0.5) -- (-0.5,-0.5);
\draw (-0.75,0.25) -- (0.25,-0.75);
\draw (-0.25,0.75) -- (0.75,-0.25);
\draw (-0.25,-0.75) -- (0.75,0.25);
\draw (-0.75,-0.25) -- (0.25,0.75);
\draw[ultra thick] (0,1) -- (0,-1);
%node at (0,1.2) {\textcolor{blue}{L}};
\filldraw[fill=red, fill opacity = 0.4, draw=red, ultra thick] (0,0) -- (0.5,0.5) -- (0,1) -- (-0.5,0.5) -- (0,0);
\filldraw[fill=blue, fill opacity = 0.4, draw=blue, ultra thick] (0,0-1) -- (0.5,0.5-1) -- (0,1-1) -- (-0.5,0.5-1) -- (0,0-1);
\filldraw[fill=violet, fill opacity = 0.4, draw=violet, ultra thick] (0,0-0.5) -- (0.5,0.5-0.5) -- (0,1-0.5) -- (-0.5,0.5-0.5) -- (0,0-0.5);
\filldraw[draw=black, ultra thick] (-1/4,1/4) -- (0,1/4);
\filldraw[draw=green, thick] (-1/4,1/4) -- (0,1/4);
\filldraw[fill=black!20 , draw=black , ultra thick] (-1/4,1/4) circle (1.7pt);
\filldraw[fill=black!60 , draw=black , ultra thick] (0,1/4) circle (1.7pt);
\end{tikzpicture}
\caption{The constraints from $H_L \setminus (\bigcup_i R_i)$ (light gray, leftmost) and $H_R$ (dark gray, rightmost), which are at distance $\ell (\frac{1}{2})^3$. As mentioned before, there is a mirrored configuration of constraints by symmetry, corresponding to $H_R \setminus (\bigcup_i R_I)$ and $H_L$.}
\end{figure}
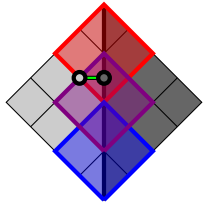
Combining all cases, we conclude that $d \geq \ell (\sqrt{3} - 1) \left(\frac{1}{2}\right)^4$, as claimed.
\end{proof}

\indent We now proceed to continue with the proof of Theorem \ref{T2}, which we divide in four main cases. These are, when $\Dcup$ intersects three, two, one, or none of the 1-cells $K^{(1)},\,K^{(2)},\,K^{(3)}$. However, before examining each case, we will note that when $\Dcup$ is contained in a 1-cell, there must exist a similarity of ratio $\frac{1}{2}$ onto $K$, making these families scaleable. By Lemma \ref{scaleable}, we will exclude these from consideration.  
\subsection*{Case 1}\label{c1}
When $\bigcup_{i = 1}^k \Delta^{(n)}_i$ intersects $K^{(1)}$, $K^{(2)}$, and $K^{(3)}$, we have that $d \geq \frac{1}{4}$. Indeed, notice there exists a projection $\pi$ onto the triangle $T$ with $d \geq |\pi(\Dcup)|$. We then provide the constraints for Case 1 in Figure \ref{fig:proj}.
\begin{figure}[ht]    
\includegraphics{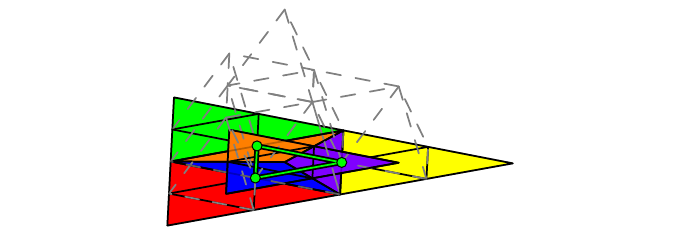}
\caption{Projection of $K$ onto $T$. By symmetry, we see that the constraints for $K^{(1)}$, $K^{(2)}$, and $K^{(3)}$ are given by $F_{31}(p_2)$, $F_{12}(p_3)$, and $F_{23}(p_1)$ respectively. These form an equilateral triangle with side length $\frac{1}{4}$.}
\label{fig:proj}
\end{figure}
By Corollary \ref{T3corPar}, $\alpha^{(1)} \geq a_n \exp{ -s_2 / 2^{n-3}}$.
    %\noindent We note that this inequality and others similar to it will be used later on to find a lower bound for $\alpha := \lim_{n \rightarrow \infty} a_n.$
    
\subsection*{Case 2}\label{c2}
Suppose $\bigcup_{i = 1}^k \Delta^{(n)}_i$ only intersects two of the base 1-cells of $K$. We will assume these are $K^{(1)}$ and $K^{(2)}$, since the other arguments follow by rotations of $\R^3$. We will provide a similar argument to Case 1, by introducing a sequence $a_n^{(2)}$. However, this case will rely on Lemma \ref{scaleable}, since $\bigcup_{i = 1}^k \Delta^{(n)}_i$ can be arbitrarily close to the point $F_1(p_2) = F_2(p_1)$ as $n \rightarrow \infty$, implying that we cannot find lower bound for $d$ unless we exclude certain scaleable families from consideration. We will show that $\alpha^{(2)} \geq a_n\exp{ -s_2 \sqrt{2}/ 2^{n-3}}$, dividing this part into two sub-cases.

\begin{figure}[h]
    \centering
    \includegraphics{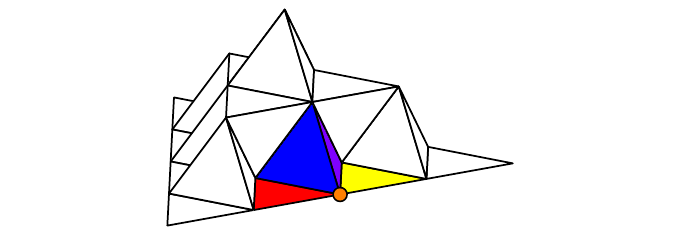}
\caption{Critical corner $R_{(2,3b)}$ for Case 2 and Subcase 3b}
\label{fig:cornercase1}
\end{figure}

\subsubsection*{Case 2a}\label{sc2.1}
If $\bigcup_{i = 1}^k \Delta^{(n)}_i$ is contained in the critical region $R_{(2,3b)} := K^{(12)} \cup K^{(43)} \cup K^{(52)} \cup K^{(21)}$, notice we can scale this corner by $2$ into $K^{(1)} \cup K^{(2)} \cup K^{(4)} \cup K^{(5)}$. By Lemma \ref{scaleable}, we may exclude this case from consideration.
\subsubsection*{Case 2b}\label{sc2.2}
Assume that $\bigcup_{i = 1}^k \Delta^{(n)}_i$ is not contained in the critical region $R_{(2)}$. By symmetry, we may suppose $\Dcup$ intersects $K^{(1)} \setminus R_{(2)}$. We then see from Figure \ref{fig:cornercase2_t3} that $d \geq \frac{\sqrt{2}}{8}$. 
\begin{figure}[h]
    \centering
    \includegraphics{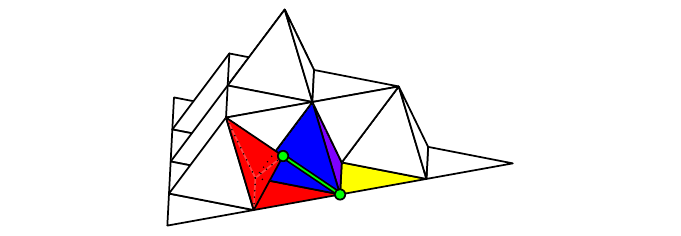}
\caption{The constraints for Case 2b corresponding to $K^{(1)} \setminus R_{(2)}$ (leftmost) and $K^{(2)}$ (rightmost) are given by $F_{51}(p_4)$ and $F_1(p_2)$ respectively, and are at distance $\frac{\sqrt{2}}{2} (\frac{1}{2})^2$.}
\label{fig:cornercase2_t3}
\end{figure}

By Corollary \ref{T3corPar}, $\alpha^{(2b)} \geq a_n^{(2b)} \exp{-s_2 / 2^{n-3}}$.
Putting $\alpha^{(2)} := \alpha^{(2b)}$ and $a_n^{(2)} := a_n^{(2b)}$, we obtain our desired sequence. Furthermore, $$\alpha^{(2)} \geq a_n^{(2)} \exp{ - s_2 \sqrt{2}/ 2^{n-3}} \geq a_n \exp{ - s_2 \sqrt{2}/ 2^{n-3}}.$$
\subsection*{Case 3}\label{c3th4}
Suppose $\bigcup_{i = 1}^k \Delta^{(n)}_i$ only intersects one of the base 1-cells of $K$. We may assume that this cell is $K^{(1)}$ by symmetry. We will subdivide this case by considering when $\Dcup$ intersects $K^{(5)}$ and $K^{(6)}$, either of these exclusively, or neither. Defining $\alpha^{(3)} := \min\{\alpha^{(3a)},\alpha^{(3b)},\alpha^{(3c)}\}$ and $a^{(3)} := \min\{a^{(3a)},a^{(3b)},a^{(3c)}\}$, we will obtain our desired sequence. Furthermore, we will show 
$$\alpha^{(3)} \geq a_n^{(3)} \exp{ -s_2 (\sqrt{2} + \sqrt{6}) / 2^{n-6}} \geq a_n \exp{ -s_2 (\sqrt{2} + \sqrt{6}) / 2^{n-6}}.$$ 
\subsubsection*{Case 3a}\label{sc3.1th4}
If $\bigcup_{i = 1}^k \Delta^{(n)}_i$ intersects both $K^{(5)}$ and $K^{(6)}$ as well, we see from the symmetries of $K$ and Figure \ref{fig:case4aT3} that $d \geq \frac{\sqrt{6}-2}{2}$.
\begin{figure}[h]
    \centering
    \includegraphics[scale=0.9]{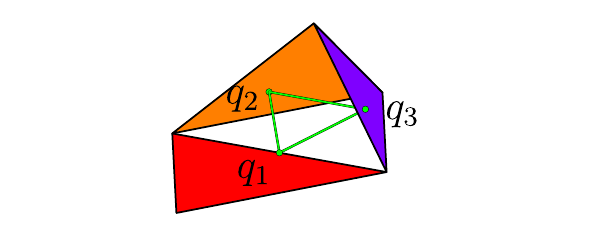}
    \caption{$K^{(1)},K^{(6)},K^{(5)}$ together with their respective constraints $q_1,q_2,q_3$, which form an equilateral triangle with side length $\frac{\sqrt{6}-2}{2}$.}
    \label{fig:case4aT3}
\end{figure}

More specifically, for $p_{i,j} := (p_i + p_j)/2$, the constraints for Case 3a seen in Figure \ref{fig:case4aT3} are given by:
\begin{align*}
q_1 &= \frac{p_{1,2} + p_{1,3}}{2},\\
q_2 &= (\sqrt{6} - 2) p_{1,3} + \frac{3 - \sqrt{6}}{2} p_{2,3} + \frac{3 - \sqrt{6}}{2}, p_4\\
q_3 &= (\sqrt{6} - 2) p_{1,2} + \frac{3 - \sqrt{6}}{2} p_{2,3} + \frac{3 - \sqrt{6}}{2} p_4.
\end{align*}
\subsubsection*{Case 3b}\label{sc3.2th4}
If $\bigcup_{i = 1}^k \Delta^{(n)}_i$ is intersects $K^{(5)}$ or $K^{(6)}$ exclusively, an argument similar to that of Case 2 holds, and we see that either $\Dcup$ is scaleable or $d \geq \frac{1}{8}$.
\begin{figure}[h]
    \centering
    \includegraphics{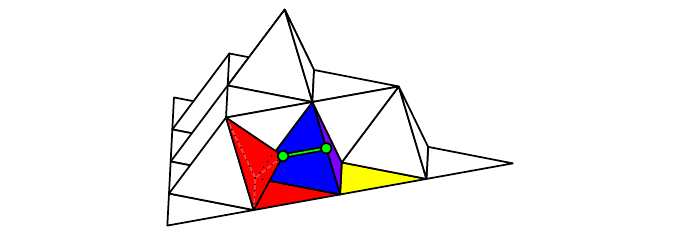}
\caption{The constraints for Case 3b corresponding to $K^{(1)} \setminus R_{(2)}$ (leftmost) and $K^{(5)}$ (rightmost) are given by $F_{51}(p_4)$ and $F_{225}(p_1)$, and are at distance $(\frac{1}{2})^3$.}
\label{fig:cornercase2}
\end{figure}
\subsubsection*{Case 3c}\label{sc3.3th4}
The difficult case arises when $\bigcup_{i = 1}^k \Delta^{(n)}_i$ is contained in $K^{(1)} \cup K^{(4)}$, since there is now a critical face $R_{(3c)} := K^{(1)} \cap K^{(4)}$ where the diameter $|\bigcup_{i = 1}^k \Delta^{(n)}_i|$ can approach $0$ as $n \rightarrow \infty$. This case must be subdivided into further sub-situations, depending on the region where $\bigcup_{i = 1}^k \Delta^{(n)}_i$ is contained.

\begin{figure}[ht]
\includegraphics{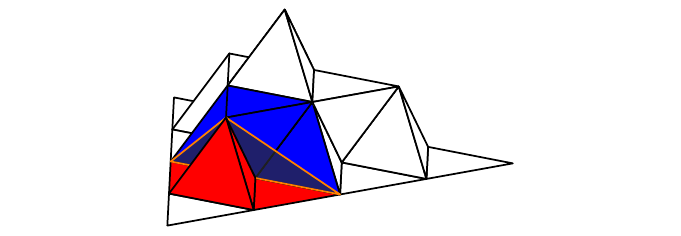}
\caption{The critical region $R_{(3c)}$ for Case 3c.} 
\label{fig:regiondiagram3.3}
\end{figure}

\noindent We first define the following regions
\begin{align*}
R_{1}' &:= K^{(321)} \cup K^{(234)} &  R_2'&:= K^{(351)} \cup K^{(234)} & R_3'&:= K^{(261)} \cup K^{(354)} & R_4'&:= K^{(231)} \cup K^{(324)}\\
    R_1 &:= K^{(21)}\cup K^{(34)} &
    R_2 &:= K^{(51)} \cup K^{(64)} &
    R_3 &:= K^{(61)} \cup K^{(54)} &
    R_4 &:= K^{(31)} \cup K^{(24)} 
\end{align*}
    $$R_5 = K^{(141)} \cup K^{(151)} \cup K^{(161)} \cup K^{(414)} \cup K^{(514)} \cup K^{(614)}$$
\begin{figure}[ht]
\includegraphics{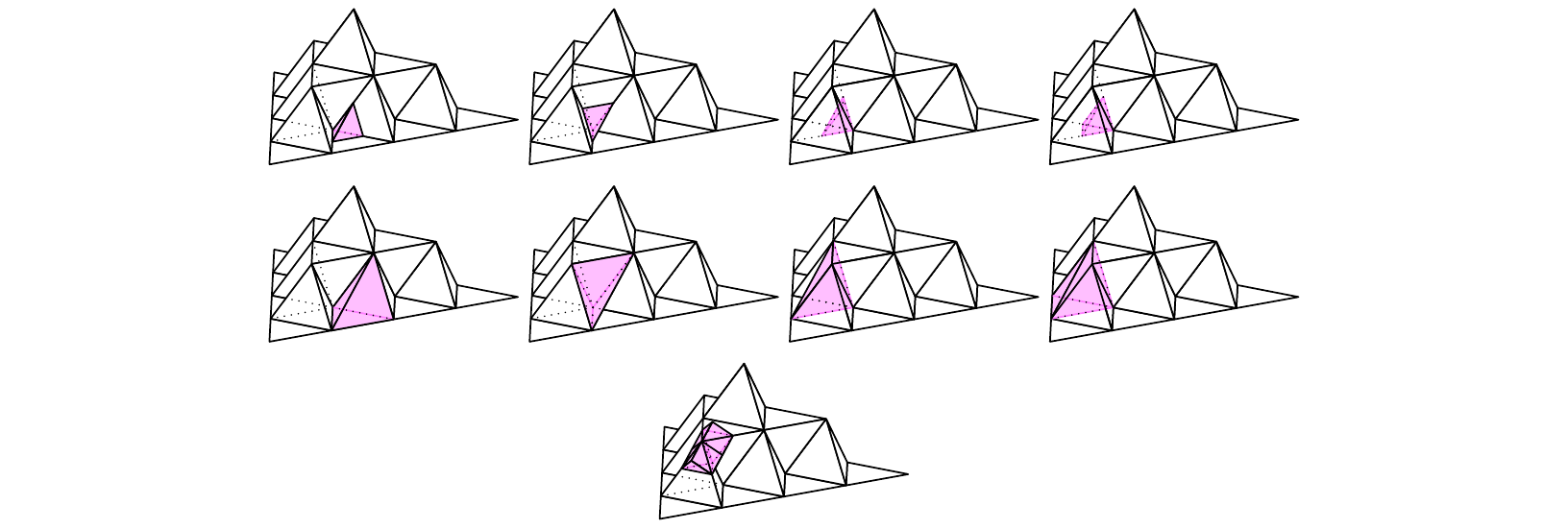}
\caption{The regions $R_1'$ through $R_4'$ (top row), $R_1$ through $R_4$ (middle row), and $R_5$ (bottom) colored in pink.}
\end{figure}
\noindent We then divide this case into the following subcases
\begin{enumerate}[label=(\roman*)]
    \item $\bigcup_i \Delta^{(n)}_i \subset R_1$, $R_2$, $R_3$, $R_4$, or $R_5$; or $\bigcup_i \Delta^{(n)}_i \subset R_1' \cup R_2' \cup R_3' \cup R_4'.$
    \item $\bigcup_i \Delta^{(n)}_i \subset R_1 \cup R_2$ or $R_3 \cup R_4$ and none of the previous cases
    \item $\bigcup_i \Delta^{(n)}_i \subset R_2 \cup R_3$ and none of the previous cases
    \item $\bigcup_i \Delta_i^{(n)} \subset R_1 \cup R_2 \cup R_3 \cup R_4$ and none of the previous cases
    \item None of the previous cases
\end{enumerate}
\medskip
\noindent \textit{Case 3c.i.} In either scenario, notice there exist similarities into $R_1 \cup R_4$ respectively of ratio $2$. By Lemma \ref{scaleable}, we may exclude this case from consideration.\medskip 

\noindent \textit{Case 3c.ii.} We now have that $\bigcup\Delta_i^{(n)} \not\subset R_1' \cup R_2' \cup R_3' \cup R_4'$. As we are trying to minimize diameter, we may suppose $\Dcup$ is contained in $R_1 \cup R_2$ or $R_3 \cup R_4$ exclusively. Furthermore, we may suppose each $\Delta_i^{(n)}$ intersects one of the two squares with diagonal $\ell = \sqrt{2} \left(\frac{1}{2}\right)^3$ in Figure \ref{fig:twosquares}. By Lemma \ref{square}, we obtain $d \geq \sqrt{2} (\sqrt{3} - 1) \left(\frac{1}{2}\right)^7$.

\begin{figure}[ht]
\includegraphics{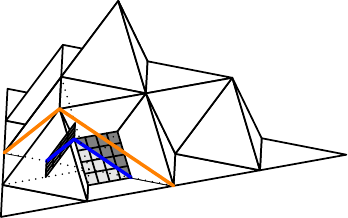}
\caption{The critical region $R_{(3c)}$ and the two squares, with diagonals of length $\sqrt{2} (\frac{1}{2})^3$ contained in the critical region, one half contained in $K^{(1)}$ and the other in $K^{(4)}$. The frontmost square is given by $(R_1 \cup R_2) \cap (R_1' \cup R_2')$ and the backmost square is given by $(R_3 \cup R_4) \cap (R_3' \cup R_4')$.}
\label{fig:twosquares}
\end{figure}
\noindent \textit{Case 3c.iii.} As we are trying to minimize diameter, we may suppose each $\Delta^{(n)}_i$ intersects the square with diameter $\ell = (\frac{1}{2})^2$ seen in Figure \ref{fig:3c.iii.square}. By Lemma \ref{square}, we obtain $d \geq (\sqrt{3} - 1) \left(\frac{1}{2}\right)^6$.

\begin{figure}[ht]
\includegraphics{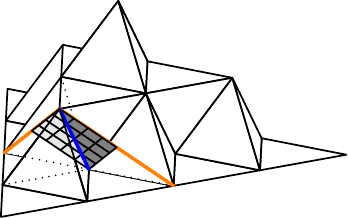}
\caption{The critical region $R_{(3c)}$ and the square contained in the critical region, given by $(R_2' \cup R_3') \cap R_{(3c)}$, with diagonal of length $(\frac{1}{2})^2$.}
\label{fig:3c.iii.square}
\end{figure}

\noindent\textit{Case 3c.iv.} We may assume that $\Dcup$ intersects $R_1$ and $R_3$, or $R_2$ and $R_4$ by exclusion of previous cases. By symmetry we may suppose $\Dcup$ intersects $R_1$ and $R_3$. From Figure \ref{fig:Case3c.iv.} we conclude $d \geq \sqrt{2} \left(\frac{1}{2}\right)^4$. 

\begin{figure}[ht]
\includegraphics[scale=1.35]{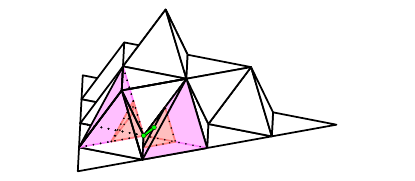}
\caption{The regions $R_1'$ and $R_3'$ (red, innermost), $R_1 \setminus R_1'$ and $R_3 \setminus R_3'$ (pink, outermost), and the constraints for Case 3c.iv.\! corresponding to $R_3$ (backmost) and $R_1 \setminus R_1'$ (frontmost), which are at distance $\sqrt{2} (\frac{1}{2})^4$. More specifically, these constraints are given by $F_{31}(p_2)$ and $F_{51}(p_4)$ respectively. There exists a mirrored configuration of constraints by symmetry, corresponding to $R_1$ and $R_3 \setminus R_3'$.}
\label{fig:Case3c.iv.}
\end{figure}

\noindent \textit{Case 3c.v.} When $\Dcup \not\subset \bigcup_{i=1}^5 R_i$, a simple calculation yields $d \geq \sqrt{2} \left(\frac{1}{2}\right)^4$. 
\begin{figure}[ht]
\includegraphics[scale=1.35]{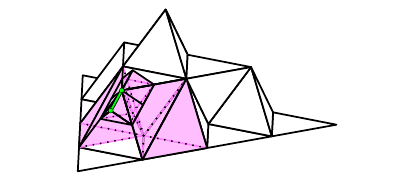}
\caption{The region $\bigcup_{i=1}^5 R_i$ colored in pink, and the constraints for Case 3c.v.\! corresponding to $R_1 \setminus (\bigcup_{i=1}^5 R_i)$ (lowermost) and $R_4$ (uppermost), which are at distance $\sqrt{2} \left(\frac{1}{2}\right)^4$. More specifically, these constraints are given by  $F_{141}(p_4)$ and $F_{4}(p_4)$ respectively. There exists a mirrored configuration of constraints by symmetry, corresponding to $R_4 \setminus (\bigcup_{i=1}^5 R_i)$ and $R_1$.}
\end{figure}

As in Case 2, we apply Corollary \ref{T3corPar} and obtain
$$\alpha^{(3)} = \alpha^{(3c)} = \alpha^{(3c.ii)} \geq a^{(3)}_n \exp{-s_2 (\sqrt{2} + \sqrt{6}) / 2^{n-6}} \geq a_n \exp{-s_2 (\sqrt{2} + \sqrt{6}) / 2^{n-6} }.$$
\subsection*{Case 4}\label{c4th4}
Suppose that $\bigcup_{i = 1}^k \Delta^{(n)}_i$ does not intersect any of the base 1-cells. We will subdivide this case by considering when $\Dcup$ intersects three or two of the 1-cells $K^{(4)}$, $K^{(5)}$, and $K^{(6)}$. Defining $\alpha^{(4)} := \min\{\alpha^{(4a)},\alpha^{(4b)}\}$ and $a^{(4)} := \min\{a^{(4a)},a^{(4b)}\}$, we obtain our desired sequence. Furthermore, $\alpha^{(4)} \geq a_n^{(4)} \exp{ -s_2 (1 + \sqrt{3}) / 2^{n-5}} \geq a_n \exp{ -s_2 (1 + \sqrt{3}) / 2^{n-5}}$. 
\subsubsection*{Case 4a}\label{sc4.1th4}
Consider when $\bigcup_{i = 1}^k \Delta^{(n)}_i$ intersects $K^{(4)},K^{(5)},K^{(6)}$. For this subcase, the critical region is the uppermost corner $R_{(4a)} := K^{(41)} \cup K^{(51)} \cup K^{(61)}$. As usual, if $\bigcup_{i = 1}^k \Delta^{(n)}_i \subseteq R_{(4a)}$, there exists a similarity into $R_4 \cup R_5 \cup R_6$ of ratio $\frac{1}{2}$, making the family case (4a)-scaleable. If $\bigcup_{i = 1}^k \Delta^{(n)}_i$ is not contained in the upper corner $R_{(4a)}$, we see from the symmetries of $K$ and Figure \ref{fig:Case4a_T3} that $d \geq \frac{\sqrt{6}-2}{4}$. This is identical to the argument in Case 3a.

\begin{figure}[ht]
\includegraphics{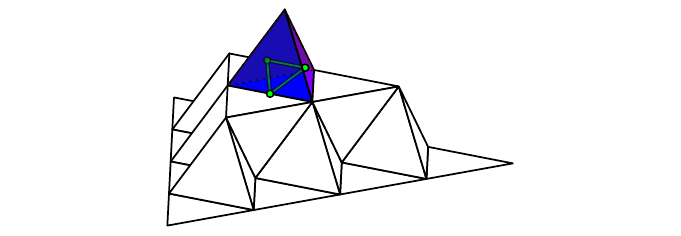}
\caption{The critical region $R_{(4a)}$ for Case 4a in color and, through $K^{(14)}$ (in blue), we find the constraints corresponding to $K^{(4)} \setminus R_{(4a)}$, $K^{(5)}$, and $K^{(6)}$. These form a triangle similar, but of scale $\frac{1}{2}$, to that formed by the constraints in Case 3a.}
\label{fig:Case4a_T3}
\end{figure}
    
 %By Lemma \ref{scaleable} we obtain that $\alpha^{(4a)} \geq a_n^{(4a)} \exp{-s_2 / 2^{n-4}}$.

\subsubsection*{Case 4b}\label{sc4.2th4}
Consider when $\bigcup_{i = 1}^k \Delta^{(n)}_i$ intersects only two of $K^{(4)},K^{(5)},K^{(6)}$. By symmetry, one can assume these are $K^{(4)}$ and $K^{(5)}$. Note that there is now a critical edge $L$ between the points $p_4$ and $F_1(p_2) = F_2(p_1)$, which is of length $\ell = \frac{1}{2}$. As we are trying to minimizing diameter, we may suppose each $\Delta_i^{(n)}$ intersects the square with diagonal $L$ seen in Figure \ref{fig:sc4.2}. An application of Lemma \ref{square} gives $d \geq (\sqrt{3} - 1) \left(\frac{1}{2}\right)^5$. 

\begin{figure}[h]
\centering
    \includegraphics{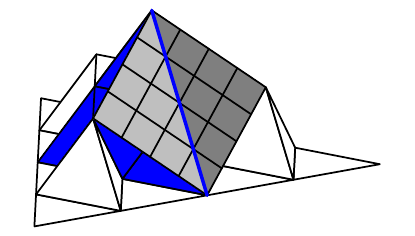}

\caption{The critical region $L$ for Case 4b depicted by a bold line (in blue), and the square with vertices $p_4$, $F_1(p_4)$, $F_2(p_1)$, and $F_2(p_4)$, with diagonal $L$ of length $\frac{1}{2}$, one side contained $K^{(4)}$ and the other in $K^{(5)}$.}
\label{fig:sc4.2}
\end{figure}

% \begin{figure}[ht]
% \centering
% \includegraphics[scale=0.06]{Case4b1.jpeg}
% \includegraphics[scale=0.06]{Case4b2.jpeg}
% \includegraphics[scale=0.06]{Case4b3.jpeg}
% \caption{Regions $R_1$, $R_2$, and $R_3$}
% \end{figure}

% \noindent\textcolor{red}{Missing three images}
We then see that
$$\displaystyle \alpha^{(4)} = \alpha^{(4b)} \geq a_n^{(4b)} \exp{ -s_2 (1 + \sqrt{3}) / 2^{n-5}} \geq a_n \exp{ -s_2 (1 + \sqrt{3}) / 2^{n-5}}.$$ 

Henceforth, combining all the above cases, we obtain
\begin{align*}
\hh^s(K) &= \lim_{n \rightarrow \infty} a_n = \lim_{n \rightarrow \infty} \min \{a_n^{(1)},a_n^{(2)},a_n^{(3)},a_n^{(4)}\} = \min \{ \alpha^{(1)},\alpha^{(2)},\alpha^{(3)},\alpha^{(4)}\}\\
&\geq a_n \min \left\{\exp{-\frac{s_2}{ 2^{n-3}}}, \exp{-\frac{s_2 \sqrt{2}}{ 2^{n-3}}}, \exp{- \frac{s_2 (\sqrt{2}+\sqrt{6})}{ 2^{n-6}}}, \exp{- \frac{s_2 (1+\sqrt{3})}{ 2^{n-5}}}\right\} \\
&= a_n \exp{-\frac{s_2 (\sqrt{2}+\sqrt{6})}{2^{n-6}}}.
\end{align*}
Combining the above inequality with Proposition \ref{p2}, we are led to the inequality (\ref{T2est}), completing the proof.

\section{Proof of Theorem 4}\label{sec6}
This demonstration will be akin to that of the proof of Theorem 3. Let $N > 2$ be such that $N \not\equiv 0$ (mod 3) and let $\{ \Delta_i^{(n)} \}_i$ be a collection of $n$-cells of $K_N$ with diameter $d := |\bigcup_{i=1}^k \Delta_i^{(n)}|$. For notational simplicity, we will denote $K := K_N$, the Koch $N$-surface. We will also organize this proof by cases, based on how many of the bottom 1-cells tangent at an edge to the peak of $K_N$ the family $\bigcup_{i = 1}^k \Delta^{(n)}_i$ intersects. In order to make this procedure as similar as possible to what was done for $K_2$, we will choose to denote the 1-cells adjacent to the ``peak'' by $K^{(1)},K^{(2)},K^{(3)}$ and those 1-cells making up the ``peak'' by $K^{(4)},K^{(5)},K^{(6)}$ (see Figure \ref{fig:kochsurf4n5second}). 

\begin{figure}[ht]
    \centering
    \includegraphics[scale=.35]{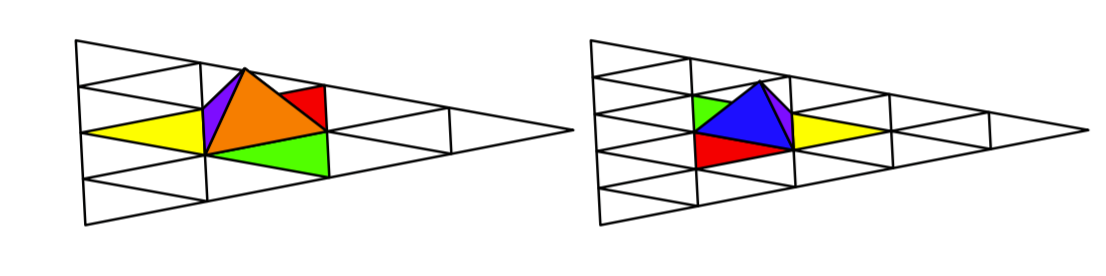}
    \caption{Koch 4-surface and 5-surface prefractals, first iteration (left and right resp.) The 1-cells $K^{(1)}$ through $K^{(6)}$ are colored in red, yellow, green, blue, purple, orange respectively (in the left diagram, $K^{(4)}$ is not visible; while in the right, $K^{(6)}$ is not visible). Notice $K^{(1)},K^{(2)},K^{(3)}$ are contained in the base triangle $T$, while $K^{(4)}, K^{(5)}, K^{(6)}$ are not as they consist of the ``peak''. Furthermore, we note that $K^{(1)}$ and $K^{(4)}$, $K^{(2)}$ and $K^{(5)}$, $K^{(3)}$ and $K^{(6)}$ each intersect at an edge.}
    \label{fig:kochsurf4n5second}
\end{figure}

Due to the difficulty in naming most of the 1-cells in $K_N$, we will depend strongly on diagrams throughout the proof, without providing formulas for the constraints these diagrams illustrate. We will again represent such constraints by green dots and graph the segments between these collections of points. The first instance of such diagrams is found in the following technical lemma, which will play an analogous role to that of Lemma 2 and will be used in Cases 3c, 4b, 4d.

\begin{lemma}\label{flap}
Consider two Koch $N$-surfaces of scale $1/N$ intersecting at a base edge of length $1/N$ forming a dihedral angle of $\theta = \arccos(1/3), 0, \arccos(1/3) - \pi$. By a base edge of the Koch $N$-surface $K$, we mean one of the three sides of length 1 in the triangle $T$.
 
\begin{figure}[ht]
    \centering
    \includegraphics[scale=0.35]{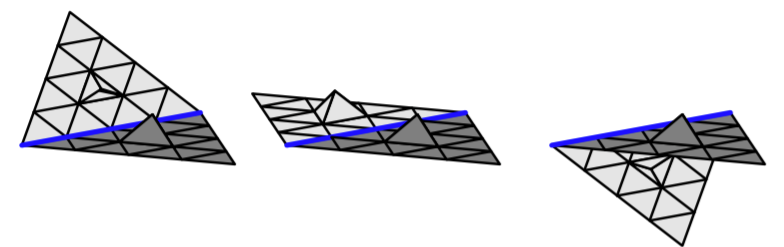}
    \label{fig:flaps}
    \caption{Two first iteration prefractals of the Koch $4$-surface, intersecting at a base edge, forming a dihedral angle of $\arccos(1/3)$ (left), $0$ (middle), and $\arccos(1/3) - \pi$. These specific angles arise in Cases 3c, 4d, and 4b respectively.}
\end{figure}
Then, if $\bigcup_{i=1}^k \Delta_i^{(n)}$ is a family of $n$-cells that intersects both Koch $N$-surfaces, it follows that $\Dset$ is scaleable or
$$d := \abs{\bigcup_{i=1}^k \Delta_i^{(n)}} \geq \frac{\sqrt{6}}{3N^3}.$$
\end{lemma}
\begin{proof}
Note that we can cover the critical edge $L$ where the two Koch $N$-surfaces by $N$ self-similar copies
of scale $1/N$ of the whole figure, see yellow-shaded regions in Figure \ref{fig:lemma3square}. We can further cover the points where adjacent yellow-shaded regions meet by $N-1$ hexagonal regions 
of scale $1/N^2$, see red-shaded regions in Figure \ref{fig:lemma3square}.
\begin{figure}[ht]
\centering
\includegraphics[scale=0.35]{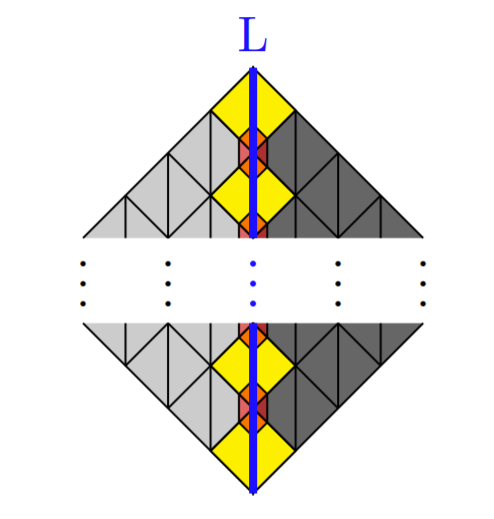}
\caption{Planar representation of two Koch $N$-surfaces meeting at an edge $L$}
\label{fig:lemma3square}
\end{figure}
Let $R$ be the region containing $L$ consising of these $N$ self-similar copies of the whole figure together with the $N-1$ hexagonal regions. We will denote these self-similar copies by $S_1,\hdots,S_N$ and the hexagonal regions by $H_1,\hdots,H_{N-1}$ such that $S_k$ and $S_{k+1}$ intersect at a point which is contained in $H_k$ for each $1 \leq k \leq N-1$.

\begin{figure}[ht]
\begin{tikzpicture}
        \filldraw[fill=yellow]   (1,-0.25) -- (0.75,0) -- (1,0.25) --  (1.25,0) -- cycle;
        \draw[blue, ultra thick] (1.25,0) -- (0.75,0);
\end{tikzpicture}
\hspace{20pt}
\begin{tikzpicture}[scale=1.3]
        \filldraw[fill=yellow]   (1/12+3/4,-1/12) -- (0.75,0) -- (1/12+3/4,1/12) --  (1/6+3/4,0) -- cycle;
        \filldraw[fill=yellow]   (-1/12+3/4,-1/12) -- (0.75,0) -- (-1/12+3/4,1/12) --  (-1/6+3/4,0) -- cycle;
        \filldraw[fill=red,fill opacity=0.5] (1/6+3/4,0) -- (1/12+3/4,1/12) -- (-1/12+3/4,1/12) -- (-1/6+3/4,0) -- (-1/12+3/4,-1/12) -- (1/12+3/4,-1/12) -- cycle;
        \draw[blue, ultra thick] (1/6+3/4,0) -- (-1/6+3/4,0);
\end{tikzpicture}
\hspace{20pt}
\begin{tikzpicture}
    \filldraw[fill=yellow]   (1,-0.25) -- (0.75,0) -- (1,0.25) --  (1.25,0) -- cycle;
    \filldraw[fill=yellow]   (0.5,-0.25) -- (0.25,0) -- (0.5,0.25) --  (0.75,0) -- cycle;
    \filldraw[fill=yellow]   (-1,-0.25) -- (-0.75,0) -- (-1,0.25) --  (-1.25,0) -- cycle;
    \filldraw[fill=yellow]   (-0.5,-0.25) -- (-0.25,0) -- (-0.5,0.25) --  (-0.75,0) -- cycle;
    \filldraw[fill=red,fill opacity=0.5] (1/4,-1/12) -- (1/12+1/4,-1/12) -- (1/6+1/4,0) -- (1/12+1/4,1/12) -- (1/4,1/12);
    \filldraw[fill=red,fill opacity=0.5] (1/6+3/4,0) -- (1/12+3/4,1/12) -- (-1/12+3/4,1/12) -- (-1/6+3/4,0) -- (-1/12+3/4,-1/12) -- (1/12+3/4,-1/12) -- cycle;
    \filldraw[fill=red,fill opacity=0.5] (-1/4,-1/12) -- (-1/12-1/4,-1/12) -- (-1/6-1/4,0) -- (-1/12-1/4,1/12) -- (-1/4,1/12);
    \filldraw[fill=red,fill opacity=0.5] (-1/6-3/4,0) -- (-1/12-3/4,1/12) -- (1/12-3/4,1/12) -- (1/6-3/4,0) -- (1/12-3/4,-1/12) -- (-1/12-3/4,-1/12) -- cycle;
    \draw[blue, ultra thick] (1.25,0) -- (0.25,0);
    \node at (0,0) {\textcolor{blue}{$\hdots$}};
    \draw[blue, ultra thick] (-0.25,0) -- (-1.25,0);
    \node at (1.45,0) {\textcolor{blue}{L}};
\end{tikzpicture}
\caption{Planar representation of a self-similar copy $S_k$ of scale $1/N$ (left), an hexagonal region $H_k$ of scale $1/N^2$ (middle), and the critical region $R$ (right)}
\end{figure}
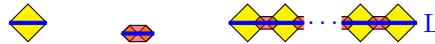

We first consider when $\bigcup_i \Delta_i^{(n)} \subseteq R$.
\begin{enumerate}
    \item If $\bigcup_i \Delta_i^{(n)}$ is contained some $S_k$, $\bigcup_i \Delta_i^{(n)}$ is scaleable.
    \item If not, suppose $\bigcup_i \Delta_i^{(n)}$ is contained in $S_k \cup H_k \cup S_{k+1}$ for some $1 \leq k \leq N-1$.
    \begin{enumerate}
        \item If $\bigcup_i \Delta_i^{(n)} \subset $
        is contained in $H_k$, then $\bigcup_i \Delta_i^{(n)}$ is scaleable.
        \item Otherwise, we see in Figure \ref{fig:twoyellowsquares} that $d \geq \sqrt{3}/2N^3$ regardless of $\theta$.
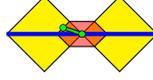
\begin{figure}[ht]
\begin{tikzpicture}[scale=2]
            \filldraw[fill=yellow]   (1,-0.25) -- (0.75,0) -- (1,0.25) --  (1.25,0) -- cycle;
            \filldraw[fill=yellow]   (0.5,-0.25) -- (0.25,0) -- (0.5,0.25) --  (0.75,0) -- cycle;
            \filldraw[fill=red,fill opacity=0.5] (1/6+3/4,0) -- (1/12+3/4,1/12) -- (-1/12+3/4,1/12) -- (-1/6+3/4,0) -- (-1/12+3/4,-1/12) -- (1/12+3/4,-1/12) -- cycle;
            \draw[blue, ultra thick] (1.25,0) -- (0.25,0);
            \draw[ultra thick] (3/4,0) -- (-1/24+3/8-1/12+3/8,1/24) ;
            \draw[green, thick] (3/4,0) -- (-1/24+3/8-1/12+3/8,1/24);
            \filldraw[black] (3/4,0) circle (.6pt);
            \filldraw[black] (-1/24+3/8,1/24) + (-1/12+3/8,0) circle (.6pt);
            \filldraw[green] (3/4,0) circle (0.45pt);
            \filldraw[green] (-1/24+3/8,1/24) + (-1/12+3/8,0) circle (0.45pt);
\end{tikzpicture}
% \begin{tikzpicture}[baseline={([yshift=-.5ex]current bounding box.center)}]
%             \filldraw[fill=yellow]   (1,-0.25) -- (0.75,0) -- (1,0.25) --  (1.25,0) -- cycle;
%             \filldraw[fill=yellow]   (0.5,-0.25) -- (0.25,0) -- (0.5,0.25) --  (0.75,0) -- cycle;
%             \filldraw[fill=red,fill opacity=0.5] (1/6+3/4,0) -- (1/12+3/4,1/12) -- (-1/12+3/4,1/12) -- (-1/6+3/4,0) -- (-1/12+3/4,-1/12) -- (1/12+3/4,-1/12) -- cycle;
%             \draw[blue, ultra thick] (1.25,0) -- (0.25,0);
%             \draw[ultra thick] (1/12+3/4,1/12) -- (-1/12+3/4,1/12);
%             \draw[green, thick] (1/12+3/4,1/12) -- (-1/12+3/4,1/12);
%             \filldraw[black] (1/12+3/4,1/12) circle (1.2pt);
%             \filldraw[black] (-1/12+3/4,1/12) circle (1.2pt);
%             \filldraw[green] (1/12+3/4,1/12) circle (0.9pt);
%             \filldraw[green] (-1/12+3/4,1/12) circle (0.9pt);
% \end{tikzpicture}

    \caption{Planar representation of the region $S_k \cup H_k \cup S_{k+1}$ for some $1 \leq k \leq N-1$, with two constraints corresponding $S_k \setminus H_k$ (leftmost) and $S_{k+1}$ (center). More specifically, the first constraint is the midpoint of one of the edges of $H_k$ which is contained in $S_k$. The second constraint is the unique point in the intersection $S_k \cap S_{k+1}$. These are at a distance of $\frac{\sqrt{3}}{2}(\frac{1}{N})^3$. An important fact to note during this computation is that, due to the planar representation, the four yellow triangles in this image are skewed when they are in fact equilateral. By symmetry, there exists a three similar configurations of constraints per value of $k$.}
\label{fig:twoyellowsquares}
\end{figure}
    \end{enumerate}
    \item If neither (1) or (2) hold, we see in Figure \ref{fig:threeyellowsquares} that $d \geq 1/N^2$ regardless of $\theta$.
\begin{figure}[ht]
\begin{tikzpicture}[baseline={([yshift=-.5ex]current bounding box.center)}]
        \filldraw[fill=yellow]   (0,-0.25) -- (-0.25,0) -- (0,0.25) --  (0.25,0) -- cycle;
        \filldraw[fill=yellow]   (-1,-0.25) -- (-0.75,0) -- (-1,0.25) --  (-1.25,0) -- cycle;
        \filldraw[fill=yellow]   (-0.5,-0.25) -- (-0.25,0) -- (-0.5,0.25) --  (-0.75,0) -- cycle;
        \filldraw[fill=red,fill opacity=0.5] (-0.25,-1/12) -- (1/12-0.25,-1/12) -- (1/6-0.25,0) -- (1/12-0.25,1/12) -- (-0.25,1/12);
        \filldraw[fill=red,fill opacity=0.5] (-1/4,-1/12) -- (-1/12-1/4,-1/12) -- (-1/6-1/4,0) -- (-1/12-1/4,1/12) -- (-1/4,1/12);
        \filldraw[fill=red,fill opacity=0.5] (-1/6-3/4,0) -- (-1/12-3/4,1/12) -- (1/12-3/4,1/12) -- (1/6-3/4,0) -- (1/12-3/4,-1/12) -- (-1/12-3/4,-1/12) -- cycle;
        \draw[blue, ultra thick] (0.25,0) -- (-1.25,0);
        \draw[ultra thick] (-0.25,0) -- (-0.75,0);
        \draw[green,thick] (-0.25,0) -- (-0.75,0);
        \filldraw[black] (-0.25,0) circle (1.3pt);
        \filldraw[black] (-0.75,0) circle (1.3pt);
        \filldraw[green] (-0.25,0) circle (1pt);
        \filldraw[green] (-0.75,0) circle (1pt);
    \end{tikzpicture}
    \caption{Planar representation of the region $S_k \cup H_k \cup S_{k+1} \cup H_{k+1} \cup S_{k+2}$ for some $1 \leq k \leq N-2$, with two constraints corresponding to $S_k$ and $S_{k+2}$. These are given by the unique points at the intersections $S_k \cap S_{k+1}$ and $S_{k+1} \cap S_{k+2}$ respectively, which are at a distance of $(\frac{1}{2})^3$.}
\label{fig:threeyellowsquares}
\end{figure}
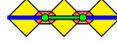
\end{enumerate}

Now we consider $\bigcup_i \Delta_i^{(n)} \not\subseteq R$.

\begin{enumerate}
    \item When $\theta = \arccos(1/3)$, we see in Figure \ref{fig:lemma3a} that $d \geq \sqrt{3}/2N^3$. 
\begin{figure}[ht]
\includegraphics[scale=1.75]{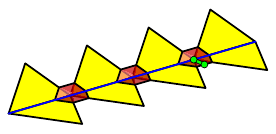}
\caption{The critical region $R$ when $\theta = \arccos(1/3)$ and $N = 4$, with two constraints corresponding to one of the two Koch $N$-surfaces (on the edge $L$) and the other Koch $N$-surface with the critical region $R$ removed (frontmost, slightly to the right). In particular, the first constraint is the unique point in $S_k \cap S_{k+1}$ for some $1 \leq k \leq N-1$. The second constraint is given by the midpoint of one of the edges of $H_k$ which are not contained in $S_k$ nor $S_{k+1}$. Due to symmetry, there are $2(N-1)$ similar configurations of constraints (in this image $N = 4$, so 6 total), two configurations per hexagon.}
\label{fig:lemma3a}
\end{figure}    

    \item When $\theta = 0$, we see in Figure \ref{fig:lemma3b} that $d \geq \sqrt{3}/2N^3$.
\begin{figure}[ht]
\includegraphics[scale=1.75]{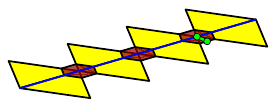}
\caption{The critical region $R$ when $\theta = 0$ and $N = 4$, with two constraints corresponding to one of the two Koch $N$-surfaces (on the edge $L$) and the other Koch $N$-surface with the critical region $R$ removed (frontmost, slightly to the right). Despite the change in geometry, these constraints are identical to those found in Figure \ref{fig:lemma3b}. Here there are also $2(N-1)$ similar configurations of constraints.}
\label{fig:lemma3b}
\end{figure}    

    \item When $\theta = \arccos(1/3) - \pi$, we see in Figure \ref{fig:lemma3c} that $d \geq \sqrt{6}/3N^3$. \qedhere
\begin{figure}[ht]
\includegraphics[scale=1.75]{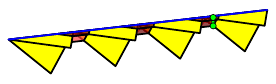}
\caption{The critical region $R$ when $\theta = \arccos(1/3) - \pi$ and $N = 4$, with two constraints corresponding to one of the two Koch $N$-surfaces (uppermost) and the other Koch $N$-surface with the critical region $R$ removed (lowermost). In particular, the second constraint is given by a point on one of the two edges of $H_k$ which are not contained in $S_k$ nor $S_{k+1}$ for some $1 \leq k \leq N-1$. The first constraint is then given by the point in $H_k$ directly above the second constraint. By translation, there exist infinitely many similar configurations of constraints.}
\label{fig:lemma3c}
\end{figure} 
\end{enumerate}
\end{proof}

We now proceed to complete the body of the proof of Theorem \ref{generalmainthm}, which again is divided into four main cases. As before, we will note that when $\Dcup$ is contained in a 1-cell, there exists a similarity of ratio $\frac{1}{N}$ onto $K$, making these families scaleable. 
By Lemma \ref{scaleable}, we will exclude these from consideration. 

\subsection*{Case 1}\label{c1th4}
When $\bigcup_{i = 1}^k \Delta^{(n)}_i$ intersects all of the distinguished base 1-cells
$K^{(1)}$, $K^{(2)}$, and $K^{(3)}$, we have that $d \geq \frac{1}{2N}$. Indeed, notice that there exists a projection $\pi$ onto the triangle $T$ with $d \geq |\pi(\Dcup)|$. As in Case 1 of Theorem 3, we use symmetry to find the three constraints as in Figure \ref{fig:Case_1_T4}, yielding $d \geq \frac{1}{2N}$. By Corollary \ref{T3corPar}, $\alpha^{(1)} \geq a_n \exp{-\frac{4s_N}{N^{n-1}}}$.
% \begin{figure}[ht]
%     \centering
%     \includegraphics[scale=0.35]{images/flatten.png}
%     \label{fig:flatten}
%     \caption{\todo{add constraints}}
% \end{figure}
\begin{figure}[ht]
    \centering
    \includegraphics{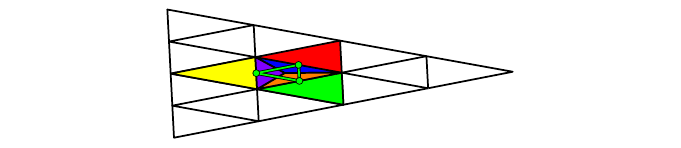}
    \caption{Projection of $K$ onto $T$ (when $N = 4$, first prefractal) with three constraints corresponding to $K^{(1)}$, $K^{(2)}$, $K^{(3)}$. These form an equilateral triangle with side length $\frac{1}{2N}$.}
    \label{fig:Case_1_T4}
\end{figure}
    %\noindent We note that this inequality and others similar to it will be used later on to find a lower bound for $\alpha := \lim_{n \rightarrow \infty} a_n.$
    
\subsection*{Case 2}\label{c2th4} 
Suppose $\bigcup_{i = 1}^k \Delta^{(n)}_i$ only intersects two of the distinguished base 1-cells of $K$. We will assume these are $K^{(1)}$ and $K^{(2)}$, since the other arguments follow by rotations of $\R^3$. We will provide a similar argument to Case 1, by introducing a sequence $a_n^{(2)}$. However, this case will rely on Lemma \ref{scaleable}, since $\bigcup_{i = 1}^k \Delta^{(n)}_i$ can be arbitrarily close to the point $F_1(p_2) = F_2(p_1)$ as $n \rightarrow \infty$, implying that we cannot find lower bound for $d$ unless we exclude Case $2$-scaleable families from consideration. We will show that $\alpha^{(2)} \geq a_n\exp{-\frac{4\sqrt{3}\,s_N}{3N^{n-2}}}$ divide this part into two sub-cases.

\begin{figure}[ht]
\centering
\includegraphics{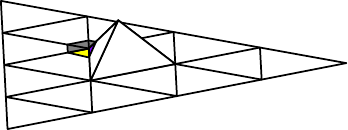}
\includegraphics{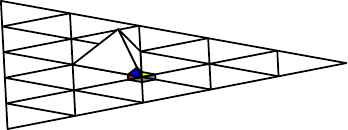}
\caption{Critical region $R_{(2,3b,4c)}$ for Cases 2, 3b, and 4c (in color), when $N = 4$ (left) and $N = 5$ (right)}
\label{fig:CornerCase2}
\end{figure}

\begin{figure}
    \centering
\includegraphics[scale=2.5]{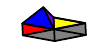}
\hspace{20pt}
\includegraphics{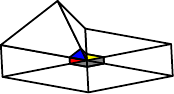}
    \caption{Critical corner $R_{(2,3b,4c)}$ for Cases 2, 3b, 4c: zoomed in (left), in $R'_{(2,3b,4c)}$ (right)}
    \label{fig:my_label}
\end{figure}

\subsubsection*{Case 2a}\label{sc2.1th4}
If $\bigcup_{i = 1}^k \Delta^{(n)}_i$ is contained in the critical region $R_{(2,3b,4c)}$, we can scale this corner by $N$ into 
$R'_{(2,3b,4c)}$.
By Lemma \ref{scaleable}, we may exclude this case from consideration.

\subsubsection*{Case 2b}\label{sc2.2th4}
Assume that $\bigcup_{i = 1}^k \Delta^{(n)}_i$ is not contained in the critical region $R_{(2,3b,4c)}$.
We see from Figure \ref{fig:sec6case2b} that $d \geq \sqrt{3}/{2N^2}$. 
\begin{figure}[ht]
\includegraphics[scale=1.5]{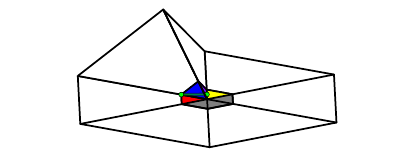}
\caption{The region $R'_{(2,3b,4c)}$ with $R_{(2,3b,4c)}$ colored, and two constraints for Case 2b corresponding to $K^{(1)} \setminus R_{(2,3b,4c)}$ (left) and $K^{(2)}$. In particular, the first constraint corresponds to the unique point on the boundary of $R_{(2,3b,4c)}$ and in the intersection $K^{(1)} \cap K^{(4)}$. The second constraint corresponds to the midpoint of the interval $K^{(2)} \cap K^{(5)} \cap R_{(2,3b,4c)}$. These are at a distance of $\frac{\sqrt{3}}{2}(\frac{1}{N})^2$. We note that by translation, there exists infinitely many similar configurations of constraints.}
\label{fig:sec6case2b}
\end{figure}

By Corollary \ref{T3corPar}, $\alpha^{(2b)} \geq a_n^{(2b)} \exp{-\frac{ 4\sqrt{3}\,s_N}{ 3N^{n-2}}}$.
Putting $\alpha^{(2)} := \alpha^{(2b)}$ and $a_n^{(2)} := a_n^{(2b)}$, we obtain our desired sequence. 
Furthermore, $$\alpha^{(2)} \geq a_n^{(2)} \exp{-\frac{ 4\sqrt{3}\,s_N}{ 3N^{n-2}}} \geq a_n \exp{-\frac{ 4\sqrt{3}\,s_N}{ 3N^{n-2}}}.$$

\subsection*{Case 3}\label{c3}
Suppose $\bigcup_{i = 1}^k \Delta^{(n)}_i$ only intersects one of the three distinguished base 1-cells 
    % \begin{lrbox}{\myasybox}
    % \includegraphics{images/main-20+0_0.pdf}
    % \end{lrbox}\raisebox{-.25\height}{\usebox{\myasybox}}
    %\hspace{5pt}
of $K$.
We may assume that this cell is 
$K^{(1)}$
by symmetry. 
We will subdivide this case by considering when $\Dcup$ intersects 
$K^{(5)}$
and
$K^{(6)}$
, either of these exclusively, or neither. 
Defining $\alpha^{(3)} := \min\{\alpha^{(3a)},\alpha^{(3b)},\alpha^{(3c)}\}$ and $a^{(3)} := \min\{a^{(3a)},a^{(3b)},a^{(3c)}\}$, we will obtain our desired sequence. 
Furthermore, we will show $\alpha^{(3)} \geq a_n \exp{-\frac{\sqrt{6}\,s_N}{N^{n-3}}}$. 
% \begin{asy}
%         settings.render=0;
%         import graph3;
%         size(100,0);
%         currentprojection=orthographic(-6,-10,4);
%         draw(surface((0.1732,0.1,0) -- (0,0.4,0) -- (-0.1732,0.1,0) -- cycle),emissive(green));
%         draw(surface((0.3464,-0.2,0) -- (0.1732,0.1,0) -- (0,-0.2,0) -- cycle), emissive(yellow));
%         draw(surface((0,-0.2,0) -- (-0.1732,0.1,0) -- (-0.3464,-0.2,0) -- cycle),emissive(red));
%         draw((0,1,0) -- (-0.866,-0.5,0) -- (0.866,-0.5,0) -- cycle);
%         draw((0.6928,-0.2,0) -- (0.5196,-0.5,0) -- (0.3464,-0.2,0) -- (0.1732,-0.5,0) -- (0,-0.2,0) -- (-0.1732,-0.5,0) -- (-0.3464,-0.2,0) -- (-0.5196,-0.5,0) -- (-0.6928,-0.2,0) -- cycle);
%         draw((0.5196,0.1,0) -- (0.3464,-0.2,0) -- (0.1732,0.1,0) -- (0,-0.2,0) -- (-0.1732,0.1,0) -- (-0.3464,-0.2,0) -- (-0.5196,0.1,0) -- cycle);
%         draw((0.3464,0.4,0) -- (0.1732,0.1,0) -- (0,0.4,0) -- (-0.1732,0.1,0) -- (-0.3464,0.4,0) -- cycle);
%         draw((0.1732,0.7,0) -- (0,0.4,0) -- (-0.1732,0.7,0) -- cycle); 
%         draw(surface((0,0,0.1633) -- (0,-0.2,0) -- (-0.1732,0.1,0) -- cycle),emissive(blue));
%         draw((0,0,0.1633) -- (0,-0.2,0) -- (-0.1732,0.1,0) -- cycle);
%         draw(surface((0.1732,0.1,0) -- (0,-0.2,0) -- (0,0,0.1633) -- cycle),emissive(purple));
%         draw((0.1732,0.1,0) -- (0,-0.2,0) -- (0,0,0.1633) -- cycle);
%     \end{asy}
\subsubsection*{Case 3a}\label{sc3.1}
If $\bigcup_{i = 1}^k \Delta^{(n)}_i$ intersects both
$K^{(5)}$
and
$K^{(6)}$, a calculation similar to when $N = 2$ reveals that $d \geq \frac{\sqrt{6}-2}{N}$.

\subsubsection*{Case 3b}\label{sc3.2}
If $\bigcup_{i = 1}^k \Delta^{(n)}_i$ intersects 
$K^{(5)}$
or
$K^{(6)}$
exclusively, an argument similar to that of Case 2 holds. Indeed, we may suppose $\Dcup$ intersects 
$K^{(5)}$ by symmetry and exclude when $\Dcup$ is contained in the critical region $R_{(2,3b,4c)}$ by Lemma \ref{scaleable}. 
    We then see from the following diagram that $d \geq \frac{\sqrt{6}}{3N^2}$.
\begin{figure}[ht]
\includegraphics[scale=1.5]{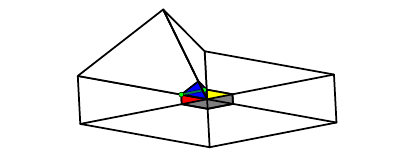}
\caption{The region $R'_{(2,3b,4c)}$ with $R_{(2,3b,4c)}$ colored, and two constraints for Case 3b corresponding to $K^{(1)} \setminus R_{(2,3b,4c)}$ (left) and $K^{(5)}$. In particular, the first constraint corresponds to the unique point on the boundary of $R_{(2,3b,4c)}$ and in the intersection $K^{(1)} \cap K^{(4)}$. The second constraint corresponds to the midpoint of the equilateral triangle $K^{(5)} \cap R_{(2,3b,4c)}$. These are at a distance of $\frac{\sqrt{6}}{3}(\frac{1}{N})^2$.}
\end{figure}

\subsubsection*{Case 3c}\label{sc3.3}

If $\bigcup_{i = 1}^k \Delta^{(n)}_i$ does not intersect
$K^{(5)}$
or
$K^{(6)}$ then $\Dcup$ intersects the two Koch $n$-surfaces $K^{(1)}$ and $K^{(4)}$ of scale $1/N$ meeting at an angle of $\theta = \arccos(1/3)$ (see Figure \ref{fig:K1nK4}). So by lemma \ref{flap}, $d \geq \frac{\sqrt{6}}{3N^3}$.

\begin{figure}[ht]
\includegraphics[scale=1.5]{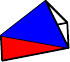}
\caption{The 1-cells $K^{(1)}$ in red (below) and $K^{(4)}$ in blue (above), which form a dihedral angle of $\arccos(1/3)$.}
\label{fig:K1nK4}
\end{figure} 
    
    By Corollary 1, $\alpha^{(3)} = \alpha^{(3c)} \geq a_n^{(3c)} \exp{-\frac{\sqrt{6}\,s_N}{N^{n-3}}} \geq a_n \exp{-\frac{\sqrt{6}\,s_N}{N^{n-3}}}$.
\subsection*{Case 4}\label{c4}
Suppose that $\bigcup_{i = 1}^k \Delta^{(n)}_i$ does not intersect any of the three distinguished base 1-cells 
    % \begin{lrbox}{\myasybox}
    % \includegraphics{images/main-35+0_0.pdf}
    % \end{lrbox}\raisebox{-.25\height}{\usebox{\myasybox}}
    of $K$.
We will subdivide this case by considering how many of the 1-cells
$K^{(4)}$, $K^{(5)}$, $K^{(6)}$
    the set $\Dcup$ intersects.
Defining $\alpha^{(4)} := \min\{\alpha^{(4a)},\alpha^{(4b)}\}$ and $a^{(4)} := \min\{a^{(4a)},a^{(4b)}\}$, we obtain our desired sequence. 
Furthermore, we show that $\alpha^{(4)} \geq a_n^{(4)} \exp{-\frac{\sqrt{6}\,s_N}{N^{n-3}}} \geq a_n \exp{-\frac{\sqrt{6}\,s_N}{N^{n-3}}}$. 
\subsubsection*{Case 4a}\label{sc4.1}
If $\bigcup_{i = 1}^k \Delta^{(n)}_i$ intersects $K^{(4)}$, $K^{(5)}$, and $K^{(6)}$, an argument similar to that of Case 2b and 3b holds. Indeed, we may exclude when $\bigcup_{i = 1}^k \Delta^{(n)}_i$ is contained in the uppermost corner $R_{(4a)}$ (see Figure \ref{fig:topcorner}) by Lemma 1. 

\begin{figure}[ht]
    \centering
    \includegraphics[scale=0.35]{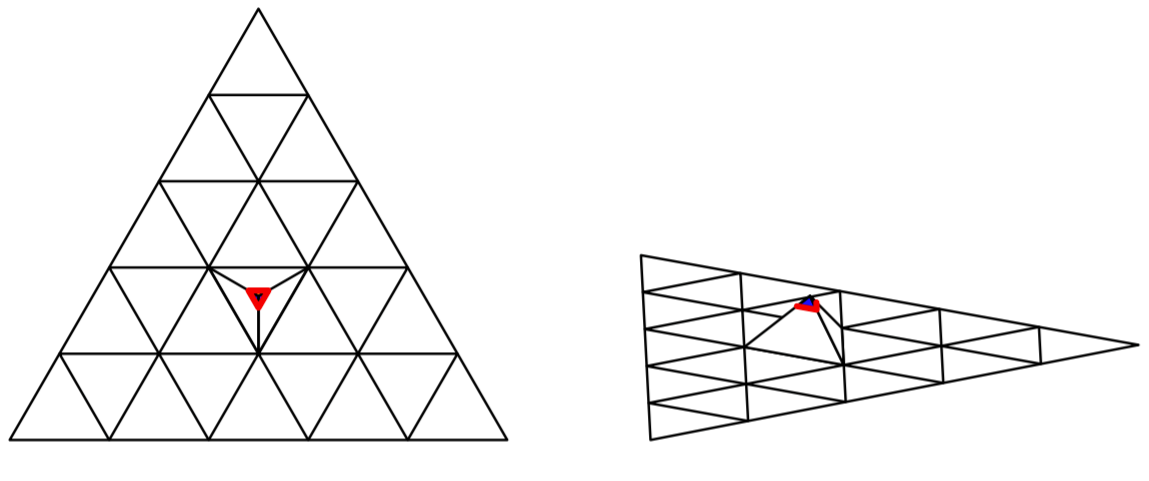} 
    \hspace{20pt}
    \raisebox{25pt}{
    \includegraphics[scale=2]{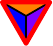}}
    \caption{The critical region $R_{(4a)}$ (when $N = 5$) from above (left), from the side (middle), and zoomed in from above (right); with a thickened red border to aid visibility. Note that $R_{(4a)}$ consists of three $2$-cells, corresponding to $K^{(4)}$, $K^{(5)}$, and $K^{(6)}$.}
    \label{fig:topcorner}
\end{figure}

Now supposing $\Dcup \not\subset R_{(4a)}$, it follows from the symmetries of $K_N$ and from the following diagram that $d \geq \frac{\sqrt{6}-2}{N^2}$. Finding this bound is identical to the calculation when $N = 2$ (see Figure \ref{Case4a_T4}).
\begin{figure}[ht]
\includegraphics{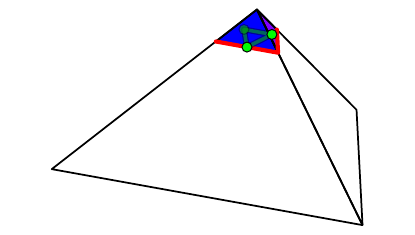}
\caption{The critical region $R_{(4a)}$ for Case 4a in color and the constraints corresponding to $K^{(4)} \setminus R_{(4a)}$, $K^{(5)}$, and $K^{(6)}$. These form a triangle similar, but of scale $\frac{2}{N^2}$, to that formed by the constraints in Theorem 3, Case 3a.}
\label{Case4a_T4}
\end{figure}

In the case where $\bigcup_{i = 1}^k \Delta^{(n)}_i \subseteq R_{(4a)}$, there exists a similarity into $K^{(4)} \cup K^{(5)} \cup K^{(6)}$ of ratio $N$, making the family Case (4a)-scaleable.
%By Lemma \ref{scaleable} we obtain that $\alpha^{(4a)} \geq a_n^{(4a)} \exp{-2\sqrt{2}s_N / N^{n-2}} \geq a_n \exp{-2\sqrt{2}s_N / N^{n-2}}$.

\subsubsection*{Case 4b}\label{sc4.2}
Consider when $\bigcup_{i = 1}^k \Delta^{(n)}_i$ intersects only two of
$K^{(4)}$, $K^{(5)}$, $K^{(6)}$, so that $\Dcup$ intersects two Koch $n$-surfaces of scale $1/N$ meeting at an angle of $\theta = \arccos(1/3) - \pi$. By lemma \ref{flap}, $d \geq \frac{\sqrt{6}}{{3}N^3}$.
\subsubsection*{Case 4c}\label{sc4.3} If $\Dcup$ intersects only one of $K^{(4)}$, $K^{(5)}$, $K^{(6)}$,
    then an identical argument to Cases 2b and 3b holds. First, we may suppose $\Dcup$ only intersects $K^{(4)}$ by symmetry. Then, we may exclude when $\Dcup$ is contained in the critical region $R_{(2,3b,4c)}$ since we may scale this corner by $N$ into 
    $R'_{(2,3b,4c)}$. 
    On the other hand, if $\Dcup$ is not contained in $R_{(2,3b,4c)}$, we see from the following diagram that $d \geq \sqrt{3}/2N^2$.
\begin{figure}[ht]
\includegraphics[scale=1.5]{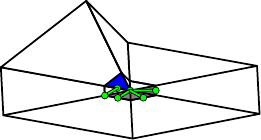}
\caption{The critical region $R_{(2,3b,4c)}$ excluding $K^{(1)}, K^{(2)}, K^{(4)}$ (in color), together with four different configurations of two constraints for Case 4c. These correspond to the regions $K^{(4)}$ and $K \setminus (K^{(1)} \cup K^{(2)} \cup K^{(4)} \cup R_{(2,3b,4c)})$ and yield a distance of $\frac{\sqrt{3}}{2} (\frac{1}{N})^2$. There exist infinitely many of these configurations by translation.}
\end{figure}

\subsubsection*{Case 4d}\label{sc4.4}
If $\Dcup$ does not intersect any of the 1-cells $K^{(4)}, K^{(5)}, K^{(6)}$, we may assume that $\Dcup$ intersects distinct adjacent 1-cells $\Delta^{(1)}_k, \Delta^{(1)}_\ell$ as we try to minimize diameter. If $\Delta^{(1)}_k$ and $\Delta^{(1)}_\ell$ intersect at an edge (forming a dihedral angle of $\theta = 0$), it follows from Lemma \ref{flap} that $d \geq \sqrt{6}/3N^3$. On the other hand, if $\Delta^{(1)}_k$ and $\Delta^{(1)}_\ell$ only intersect at a point $p$ then an argument similar to Case 2b and 3b holds. Indeed, consider the region $R_{(4d)}'$ consisting of the six 1-cells which intersect at $p$. Note that there exists a similar region $R_{(4d)}$ of six $2$-cells which intersect at $p$. 
\begin{figure}[h]
    \centering
    \includegraphics{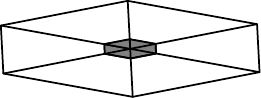}
    \caption{The region $R_{(4d)}$ (in gray) inside $R_{(4d)}'$}
\end{figure}

Now if $\Dcup \subset R'$, we can scale $R_{(4d)}$ by a factor of $N$ into $R_{(4d)}'$. So by Lemma \ref{scaleable}, we may exclude this case from consideration. When $\Dcup \not\subset R_{(4d)}$, we see from the following diagram that $d \geq \sqrt{3}/2N^2$.
\begin{figure}[h]
    \centering
    \includegraphics{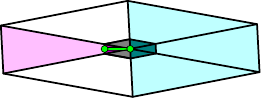}
    \caption{The region $\Delta_k^{(1)}$ in pink (left) with the three possible choices for $\Delta_\ell^{(1)}$ in blue (right), together with a configuration of two constraints corresponding to $\Delta_k^{(k)}$ (left) and a possible $\Delta_\ell^{(1)}$ (right). These are at a distance of $\frac{\sqrt{3}}{2} (\frac{1}{N})^2$. Note that there are infinitely many such configurations by translation.}
\end{figure}

We then see that
$$\alpha^{(4)} = \alpha^{(4b)} \geq a_n^{(4b)} \exp{-\frac{\sqrt{6}\,s_N}{N^{n-3}}} \geq a_n \exp{-\frac{\sqrt{6}\,s_N}{N^{n-3}}}.$$

Henceforth, combining all the above cases, we obtain
\begin{align*}
\hh^s(K) &= \displaystyle\lim_{n \rightarrow \infty} a_n = \displaystyle\lim_{n \rightarrow \infty} \min \{a_n^{(1)},a_n^{(2)},a_n^{(3)},a_n^{(4)}\} = \min \{ \alpha^{(1)},\alpha^{(2)},\alpha^{(3)},\alpha^{(4)}\} \\   
&\geq a_n \min \left\{ \exp{-\frac{4s_N}{N^{n-1}}},\, \exp{-\frac{4\sqrt{3}\,s_N}{3N^{n-2}}},\, \exp{-\frac{\sqrt{6}\,s_N}{N^{n-3}}}\right\}
\\
&= a_n \exp{-\frac{\sqrt{6}\,s_N}{N^{n-3}}}.
\end{align*}

Combining the above inequality with Proposition \ref{p2}, we are lead to the inequality (\ref{T2est}), completing the proof.

\section{Application}\label{sec7}

\indent We now present an application to the theory of Partial Differential Equations over the Koch $N$-crystals.

\subsection{Preliminaries}\label{subsec7.1}

To begin, denote by $\Omega_{_N}:=\textrm{int}(\mathcal{C}_N)\subseteq\mathbb{R\!}^{\,3}$ the interior of the Koch $N$-crystal defined by Definition \ref{3.2} with boundary $\Gamma_{_N}:=K_N$ (as in Definition \ref{kochNsurface}), for $N>2$ with $N\not\equiv 0\,(\bmod\,\,3)$. As discussed at the end of section \ref{sec3}, one has that $\Omega_{_N}$ is an uniform domain whose boundary $\Gamma_{_N}$ is a $s_N$-set with respect to the $s_N$-dimensional Hausdorff measure, where $s_N:=\log (N^2 + 2) / \log(N)$. Given a set $E$, we denote by $L^p(E):=L^p(E,\mu)$ the $L^p$-based space of $\mu$-measurable functions, and we write $\lambda_3$ as the $3$-dimensional Lebesgue measure. Also, for a domain $D$, we denote by $W^{1,p}(D)$ the well-known $p$-Sobolev spaces, for $1\leq p\leq\infty$. When $p=2$, we write $H^1(D):=W^{1,2}(D)$. Finally, by $\overline{\Omega}_{_N}$ we denote the closure of the set $\Omega_{_N}$.\\
\indent For non-Lipschitz domains, a normal derivative may not be well-defined. This is a key point when defining a Robin problem over irregular regions.
However, define the Robin-type bilinear form $(\mathcal{E},D(\mathcal{E}))$ by: $D(\mathcal{E}):=(H^1(\Omega_{_N})\cap C(\overline{\Omega}_{_N}))\times(H^1(\Omega_{_N})\cap C(\overline{\Omega}_{_N}))$, and
\begin{equation}\label{bil-form}
\mathcal{E}(v,w):=\displaystyle\int_{\Omega_{_N}}\nabla v\nabla w\,dx+\displaystyle\int_{\Gamma_{_N}}\beta vw\,d\mathcal{H}^{s_N},
\end{equation}
for $\beta\in L^{\infty}(\Gamma_{_N})$ with
$\displaystyle\textrm{ess}\,{\inf_{x\in\Gamma_{_N}}}\beta(x)\geq b_0$ for some constant $b_0>0$. Then as $\mathcal{H}^{s_N}$ is a $s_N$-Ahlfors measure, it follows from \cite[Remark 6.5]{BIE09}
that $\mathcal{H}^{s_N}$ fulfills the conditions (11) in \cite{AR-WAR03}, and consequently from \cite[Theorem 3.3]{AR-WAR03}, one has that the form
$(\mathcal{E},D(\mathcal{E}))$ is closable, which means that the corresponding Robin problem is well-posed over this domain. Furthermore, it is shown in \cite{BIE09} that there exists a compact trace map from $H^1(\Omega_{_N})$ into $L^2(\Gamma_{_N})$.
Thus, following this argument, we define an appropriate interpretation of the normal derivative over $s$-sets, whose motivation and general
form is taken from \cite{BIE-WAR}.

\begin{definition}\label{def-gen-n-d}
Let $F:\Omega_{_N}\rightarrow\mathbb{R\!}^{\,3}$ be a measurable function.
If there exists $f\in L^1_{loc}(\mathbb{R\!}^{\,3})$ with
$$\int_{\Omega_{_N}}F\nabla v\,dx=\int_{\Omega_{_N}}fv\,dx+\int_{\Gamma_{_N}}v\,d\mathcal{H}^{s_N},$$
for all $v\in C^1_c(\overline{\Omega}_{_N})$, then we say that $\mathcal{H}^{s_N}$ is the \textbf{normal measure} of $F$, and we write
$$N^{\ast}_s(F):=\mathcal{H}^{s_N}.$$ Note that when the normal measure $N^{\ast}_s(F)$ exists, it is unique,
and one has that $dN^{\ast}_s(\phi F)=\phi dN^{\ast}_s(F)$ for all $\phi\in C^1(\overline{\Omega}_{_N})$.
Therefore, if $u\in W^{1,1}_{loc}(\Omega_{_N})$ and $\nabla u$ exits, then we denote by
$$\frac{\partial u}{\partial\nu_{s}}:=N^{\ast}_s(\nabla u)$$ de \textbf{generalized normal derivative} of $u$ in $\Gamma_{_N}$.
\end{definition}

\indent Observe that if $\Omega_{_N}\subseteq\mathbb{R\!}^{\,n}$ is a bounded Lipschitz domain and $s_N=n-1$, then
Definition \ref{def-gen-n-d} agrees with the classical definition of a normal derivative. In the case of non-Lipschitz domains
Definition \ref{def-gen-n-d} is a sort of interpretation of a normal derivative, which we use to write a well-posed Robin problem
over irregular regions (as explained above).\\

\begin{remark}\label{remark-generalized-normal-derivative}
Since $\Omega_{_N}\subseteq\mathbb{R\!}^{\,3}$ is a bounded extension domain (in the sense of Jones \cite{JON}) whose boundary
$\Gamma_{_N}$ is an $s_N$-set with respect to the Hausdorff measure $\mathcal{H}^{s_N}$, then one can follow the approach given by
Hinz, Rozanova-Pierrat and Teplyaev \cite{HINZ-ROZ-TEP22,HINZ-ROZ-TEP21} to give sense to the normal derivative over non-Lipschitz domains.
To be more precise, in \cite[Theorem 7]{HINZ-ROZ-TEP21}, the authors
established the existence of the generalized normal derivative $\frac{\partial u}{\partial\nu_{s}}:=N^{\ast}_s(\nabla u)$ of $u$ over $\Gamma_{_N}$
as the linear bounded functional $\frac{\partial u}{\partial\nu_{s}}\in\left(Tr_{_{\Gamma_{_N}}}(H^1(\mathbb{R\!}^{\,3}))\right)^{\ast}$, given by
$$\left\langle\frac{\partial u}{\partial\nu_{s}},Tr_{_{\Omega_{_N},\Gamma_{_N}}}v\right\rangle_{_{\left(Tr_{_{\Gamma_{_N}}}(H^1(\mathbb{R\!}^{\,3}))\right)^{\ast},
Tr_{_{\Gamma_{_N}}}(H^1(\mathbb{R\!}^{\,3}))}}=\displaystyle\int_{\Omega_{_N}}v\Delta u\,dx+\displaystyle\int_{\Omega_{_N}}\nabla u\nabla v\,dx,$$
for all $u,\,v\in H^1(\Omega_{_N})$, provided that $\Delta u\in L^2(\Omega_{_N})$, where $Tr_{_{\Gamma_{_N}}}:H^1(\mathbb{R\!}^{\,m})\rightarrow
L^2(\Gamma_{_N})$ denotes a bounded trace operator defined as in \cite[equalities (14) and (13)]{HINZ-ROZ-TEP21},
and $Tr_{_{\Omega_{_N},\Gamma_{_N}}}:=Tr_{_{\Gamma_{_N}}}\circ S_{_{\Omega_{_N}}}$, for $Tr_{_{\Omega_{_N}}}:H^1(\mathbb{R\!}^{\,3})\rightarrow
L^2(\Omega_{_N})$ a bounded operator, and $S_{_{\Omega_{_N}}}:H^1(\Omega_{_N})\rightarrow H^1(\mathbb{R\!}^{\,3})$ the bounded extension operator
(e.g. \cite[Theorem 7]{HINZ-ROZ-TEP21}). One has that Definition \ref{def-gen-n-d} is a more general interpretation of a normal
derivative (since it can include domains which may not have the extension property), but the latter formulation
gives a more concrete structure for the normal derivative over classes of bounded domains with (possibly) irregular boundaries.
Since it is not the intention of this paper to go deeper into these subjects, we do not go into further details.
\end{remark}

\indent We now present the following example.

\subsection{The Robin boundary value problem}\label{subsec7.2}

\indent Consider now the realization of the following boundary value problem:
\begin{equation}
\label{1.01}\left\{
\begin{array}{lcl}
-\Delta u\,=\,f\,\,\,\,\,\,\,\,\,\indent\indent\indent\indent\textrm{in}\,\,\,\Omega_{_N},\\[1ex]
\displaystyle\frac{\partial u}{\partial\nu_{s_N}}+\beta u\,=\,g\,\,\,\,\,\indent\indent\,\,\,\textrm{on}\,\,\,\Gamma_{_N}\,,\\
\end{array}
\right.
\end{equation}\\
for $f\in L^p(\Omega_{_N})$,\, $g\in L^q(\Gamma_{_N})$, and $\beta\in L^{\infty}(\Gamma_{_N})$ with
$\displaystyle\textrm{ess}\,{\inf_{^{x\in\Gamma_{_N}}}}\beta(x)\geq b_0$ for some constant $b_0>0$. Then equation (\ref{1.01}) turns out to be a Robin boundary value problem over the Koch $N$-crystal. In fact, in \cite{VELEZ2013-1} it is shown that one can define the Robin problem over more classes of irregular domains, in which the Koch $N$-crystals form family of domains fulfilling the required properties.\\
\indent A function $u\in H^1(\Omega_{_N})$ is said to be a {\bf weak solution} of the Robin problem (\ref{1.01}), if
$$\mathcal{E}_{_N}(u,\varphi)=\displaystyle\int_{\Omega_{_N}}f\varphi\,dx+\displaystyle\int_{\Gamma_{_N}}g\varphi\,d\hh^{s_N}$$
for all $\varphi\in H^1(\Omega_{_N})$, where $\mathcal{E}_{_N}(\cdot,\cdot):=\mathcal{E}(\cdot,\cdot)$ denotes the bilinear form defined by
(\ref{bil-form}), for $D:=\Omega_{_N}$,\, $\partial D:=\Gamma_{_N}$, and $\eta:=\hh^{s_N}$.\\

\indent Since $\Omega_{_N}$ is an uniform domain, $\Gamma_{_N}$ is a $s_N$-set with respect to $\hh^{s_N}$, and $s_N\in(1,3)$,
recalling \cite[Theorem 4.24]{BIE10.1}, one gets that the form $\mathcal{E}_{_N}(\cdot,\cdot)$ is continuous, and coercive. Moreover,
the conclusions in \cite{VELEZ2013-1} imply the following important result.

\begin{theorem}
{\it If $f\in L^p(\Omega_{_N})$ and $g\in L^q(\Gamma_{_N})$ for $p>\frac{3}{2}$ and $q>\frac{s_N}{s_N-1}$, then
the Robin problem (\ref{1.01}) admits a unique weak solution $u\in H^1(\Omega_{_N})$, and there exists a constant
$\delta\in(0,1)$ such that $u\in C^{0,\delta}(\overline{\Omega}_{_N})$, that is, $u$ is globally H\"older continuous. Furthermore, there is a constant $C>0$ (independent of $u$), such that}
$$\|u\|_{_{C^{0,\delta}(\overline{\Omega}_{_N})}}\,\leq\,C\left(\|f\|_{_{p,\Omega_{_N}}}+\|g\|_{_{q,\Gamma_{_N}}}\right).$$
\end{theorem}

\indent The above result is also valid in the quasi-linear case involving the $p$-Laplace operator, for $6(s_N+2)^{-1}<p<\infty$ (under some modifications on $q$ and $r$; see \cite{VELEZ2013-1}). Recently a generalization of this result has been obtained to the Robin problem involving variable exponents and anisotropic structures. For more details, refer to \cite{BOU-AVS18}.\\
\indent The above results can be considered as generalizations of results obtained in \cite{AV-NYS14,LEW-NYS12,NYS96,NYS94}, where regularity results were obtained
for the Dirichlet problem over classes of non-Lipschitz domains. However, it is important to point out that most of the results in
\cite{AV-NYS14,LEW-NYS12,NYS96,NYS94} were developed over bounded NTA domains whose boundaries are $(n-1)$-sets, while in our case we are allowing
the boundary to be an $s_N$-set for $s_N\in(1,3)$. It is important to mention that NTA domains include the classical Koch snowflake domains and other fractal-like domains whose Hausdorff dimensions may not be $n-1$. However, on the works in \cite{AV-NYS14,LEW-NYS12,NYS96,NYS94}, the authors are assuming the Ahlfors-David regular condition over the boundary, which is equivalent to say that the boundary of the domain has Hausdorff dimension of $n-1$, with such boundary being an $(n-1)$-set. Furthermore, for Robin-type boundary value problems, usually one needs more ``geometrical structure" in
the domains under consideration, in order to have a notion of a normal derivative (as explained at the beginning of this section), and for the
well-posedness.

\section*{Acknowledgements}

\indent We would like to thank Ernesto Ferrer for his help in modeling the first and second iterations of prefractals for the Koch 4-crystal and 5-crystal, seen in Figure \ref{fig:crystals}.
We also thank the referees for their careful read through the paper, and their helpful comments and suggestions.\\[5pt]
\indent The second author is supported by The Puerto Rico Science, Technology and Research Trust (PR-Trust) under agreement number 2022-00014.

\smallskip
\indent\\
\noindent{\bf Disclaimer.}  {\it This content is only the responsibility of the authors and does not necessarily represent the official views of The Puerto Rico Science, Technology and Research Trust}.\\


\begin{thebibliography}{9}
\bibitem{AR-WAR03}
W.~Arendt and M.~Warma.
\newblock{The {L}aplacian with {R}obin boundary conditions on arbitrary domains}.
\newblock{\em Potential Analysis} {\bf 19} (2003),  341--363.

\bibitem{AV-NYS14}
B.~Avilen and K.~Nystr\"om.
\newblock{Estimates for solutions to equations of $p$-Laplace type in Ahlfors regular NTA-domains}.
\newblock{\em J. Functional Analysis} {\bf 266} (2014),  5955--6005.

\bibitem{BIE10.1}
M.~Biegert.
\newblock{A priori estimate for the difference of solutions to quasi-linear elliptic equations}.
\newblock{\em Manuscripta Math.} {\bf 133} (2010), 273--306.

\bibitem{BIE09}
M.~Biegert.
\newblock{On trace of {S}obolev functions on the boundary of extension domains}.
\newblock{\em Proc. Amer. Math. Soc.} {\bf 137} (2009), 4169--4176.

\bibitem{BIE-WAR}
M.~Biegert and M.~Warma.
\newblock{Some quasi-linear elliptic Equations with inhomogeneous generalized Robin boundary conditions on ``bad" domains}.
\newblock{\em Advances in Differential Equations} {\bf 15} (2010), 893--924.

\bibitem{BOU-AVS18} M.-M.~Boureanu and A.~V{\'e}lez\,-\,Santiago.
\newblock {Fine regularity for elliptic and parabolic anisotropic Robin problems with variable exponents}.
\newblock{\it J. Differential Equations} {\bf 266}  (2019), 8164--8232.

\bibitem{Teplyaev2022}
A.~Dekkers, A.~Rozanova-Pierrat, A.~Teplyaev.
\newblock{Mixed boundary valued problems for linear and nonlinear wave equations in domains with fractal boundaries},
\newblock{\em Calc. Var. PDEs} {\bf 61}, 75 (2022).

\bibitem {falconer1}
K.~J.~Falconer.
\newblock{\em Fractal geometry: Mathematical foundations and applications}.
\newblock{Chichester: Wiley}, 1990.

\bibitem{HINZ-ROZ-TEP22}
M.~Hinz, A.~Rozanova\,-\,Pierrat, A.~Teplyaev.
\newblock{Boundary value problems on non-Lipschitz domains: stability, compactness and the existence of optimal shapes}.
\newblock{\em Assimptotic Analysis, to appear} (2023).

\bibitem{HINZ-ROZ-TEP21}
M.~Hinz, A.~Rozanova\,-\,Pierrat, A.~Teplyaev.
\newblock{Non-Lipschitz uniform domain shape optimization in linear acustics}.
\newblock{\em SIAM J. Control Optimization} {\bf 59} (2021), 1007--1032.

\bibitem{HUT81}
J.~E.~Hutchinson.
\newblock{Fractals and self-similarity}.
\newblock{\em Indiana Univ. Math}. {\bf 30}  (1981), 713--747.

\bibitem{bao1}
B.~Jia.
\newblock{Bounds of Hausdorff measure of the Sierpinski gasket}.
\newblock{\em J. Math. Anal. Appl}. {\bf 330}  (2007), 1016--1024.

\bibitem {bao2}
B.~Jia.
\newblock{Bounds of the Hausdorff measure of the Koch curve}.
\newblock{\em Applied Mathematics and Computation} {\bf 190}  (2007), 559--565.

\bibitem{JON}
P.~W.~Jones.
\newblock{Quasiconformal mappings and extendability of functions in Sobolev spaces}.
\newblock{\em Acta Math}. {\bf 147}  (1981), 71--88.

\bibitem{LAN-VELEZ-VER18-1}
M.~R.~Lancia, A.~V{\'e}lez\,-\,Santiago, P.~Vernole.
\newblock{A quasi-linear nonlocal Venttsel' problem of Ambrosetti--Prodi type on fractal domains}.
\newblock{\em Discrete \& Continuous Dynamical Systems - Series A} {\bf 39} (2019), 4487--4518.

\bibitem{LAN-VELEZ-VER15-1}
M.~R.~Lancia, A.~V{\'e}lez\,-\,Santiago, P.~Vernole.
\newblock{Quasi-linear {V}enttsel' problems with nonlocal boundary conditions on fractal domains}.
\newblock{\em Nonlinear Analysis: Real World Applications} {\bf 35} (2017), 265--291.

\bibitem{LAP-PANG95}
M.~L.~Lapidus and M.~Pang.
\newblock{Eigenfunctions of the {K}och snowflake domain}.
\newblock{\em Commun. Math. Phys.} {\bf 172}  (1995), 359--376.

\bibitem{LEW-NYS12}
J.~L.~Lewis and K.~Nystr\"om.
\newblock{Regularity and free boundary regularity for the $p$-Laplace operator in Reifenberg flat and Ahlfors regular domains}.
\newblock{\em J. Amer. Math. Soc.} {\bf 25} (2012),  827--862.

\bibitem{MAT}
P.~Mattila.
\newblock{\em Geometry of Sets and Measures in Euclidean Spaces}.
\newblock{Cambridge Univ. Press}, 1995.

\bibitem{NYS96}
K.~Nystr\"om.
\newblock{Integrability of Green potentials in fractal domains}.
\newblock{\em Arkiv f\"or Matematik} {\bf 34} (1996),  335--381.

\bibitem {NYS94}
K.~Nystr\"om.
\newblock{\em Smoothness properties of solutions to Dirichlet problems in domains with fractal boundary}.
\newblock{Doctoral Thesis, University of Ume\"a, Ume\"a} (1994).

\bibitem{VELEZ2013-1}
A.~V{\'e}lez\,-\,Santiago.
\newblock{Global regularity for a class of quasi-linear local and nonlocal elliptic equations on extension domains}.
\newblock{\em J. Functional Analysis} {\bf 269}  (2015),  1--46. 


\end{thebibliography}
\end{document}